\documentclass[11pt]{amsart}
  \usepackage[ngerman,english]{babel}

  \usepackage[latin1]{inputenc}
  \usepackage[
  pdfstartpage=1,
  pdfauthor={Timo Essig},
  pdfsubject={Intersection Space Cohomology of Three-Strata Pseudomanifolds},
  pdfkeywords={Intersection Space Cohomology, Pseudomanifolds, Poincar\'e Duality, Differential Forms},
  breaklinks=true,
  colorlinks=true,
  linkcolor=black,
  anchorcolor=black,
  citecolor=black,
  filecolor=black,
  menucolor=black,
  pagecolor=black,
  urlcolor=black
  ]
  {hyperref}
  \usepackage{amsmath}
  \usepackage{amssymb}
  \usepackage{accents}
  \usepackage{enumerate}

  \usepackage{graphicx}
  \usepackage{tikz-cd}
  \usetikzlibrary{decorations.pathreplacing}
  \usetikzlibrary{calc}

  \usepackage[arrow, matrix, curve]{xy}

  \usepackage{afterpage}
  \usepackage{subfig}
  \usepackage{setspace}

  \newtheorem{thm}{Theorem}[subsection]
  \newtheorem{lemma}[thm]{Lemma}
  \newtheorem{prop}[thm]{Proposition}
  \newtheorem{rem}[thm]{Remark}
  \newtheorem{cor}[thm]{Corollary}
  
  \newtheorem{ex}[thm]{Example}

  \newtheorem{mydef}[thm]{Definition}

  \makeatletter
  \@addtoreset{thm}{section}
  
  \providecommand*{\xhookrightfill@}{%
   \arrowfill@{\lhook\joinrel\relbar}\relbar\rightarrow
  }
  \providecommand*{\xhookrightarrow}[2][]{%
   \ext@arrow 0395\xhookrightfill@{#1}{#2}%
  }
  \makeatother

  \DeclareMathOperator{\cone}{cone}
  \DeclareMathOperator{\id}{id}
  \DeclareMathOperator{\im}{im}

  \newcommand{\MS}{\mathcal{MS}}
  \newcommand{\del}{\partial}
  \newcommand{\pp}{\bar{p}}
  \newcommand{\qq}{\bar{q}}

  \newcommand{\OMS}{\Omega_{\MS}^\bullet}
  \newcommand{\OMSc}{\Omega_{\MS,c}^\bullet}
  
  \newcommand{\OEMS}{\Omega_{E \MS}^\bullet}
  \newcommand{\ftsOMS}{ft_{<K} \Omega_{\MS}^\bullet}
  \newcommand{\ftgOMS}{ft_{\geq K} \Omega_{\MS}^\bullet}
  \newcommand{\ftgqOMS}{ft_{\geq K^*} \Omega_{\MS}^\bullet}
  \newcommand{\ftgOMSc}{ft_{\geq K} \Omega_{\MS, c}^\bullet}
  \newcommand{\ftgqOMSc}{ft_{\geq K^*} \Omega_{\MS, c}^\bullet}

  \newcommand{\OI}{\Omega I_{\bar{p}}^\bullet}
  
  \newcommand{\OIq}{\Omega I_{\bar{q}}^\bullet}
  \newcommand{\wOI}{\widetilde{\Omega I}_{\bar{p}}^\bullet}
  \newcommand{\wOIq}{\widetilde{\Omega I}_{\bar{q}}^\bullet}

  \newcommand{\dM}{\partial M}
  \newcommand{\dE}{\partial E}
  \newcommand{\dW}{\partial W}
  \newcommand{\dB}{\partial B}

  \newcommand{\dQ}{\partial Q}

  \newcommand{\R}{\mathbb{R}}
  
  \newcommand{\Z}{\mathbb{Z}}
  \newcommand{\N}{\mathbb{N}}
  \newcommand{\K}{\mathbb{K}}

  \newcommand{\BE}{\hfill \( \square \)}

  \newcommand{\proj}{\text{proj}}

  \newcommand{\tg}{\tau_{\geq L} }
  \newcommand{\tgq}{\tau_{\geq L^*} }
  \newcommand{\ts}{\tau_{<L}}

\begin{document}

  \title[Extending Intersection Space Cohomology]{Intersection Space Cohomology of Three-Strata Pseudomanifolds}
  \author{J. Timo Essig}
  \address{Karlsruher Institut f\"ur Technologie, Adenauerring 2, 76131 Karlsruhe}
  \email{timo.essig@kit.edu}

  \date{February, 2018}
  \keywords{Singularities, Stratified Spaces, Pseudomanifolds, Poincar\'e Duality, Intersection Cohomology, de Rham Theory, Differential Forms, Manifolds with Corners, Deformation of Singularities, Mirror Symmetry, Scattering Metric\\
The author was in part supported by a grant of the Evangelische Studienstiftung.}

\begin{abstract}
The theory of intersection spaces assigns cell complexes to certain stratified topological pseudomanifolds depending on a perversity function in the sense of intersection homology. The main property of the intersection spaces is Poincar\'e duality over complementary perversities for the reduced singular (co)homology groups with rational coefficients. This (co)homology theory is not isomorphic to intersection homology, instead they are related by mirror symmetry. Using differential forms, Banagl extended the intersection space cohomology theory to 2-strata pseudomanifolds with a geometrically flat link bundle. In this paper we use differential forms on manifolds with corners to generalize the intersection space cohomology theory to a class of 3-strata spaces with flatness assumptions for the link bundles. 
We prove Poincar\'e duality over complementary perversities for the cohomology groups. To do so, we investigate fiber bundles on manifolds with boundary.
At the end, we give examples for the application of the theory.
\end{abstract}

	
\tikzstyle{reverseclip}=[insert path={(current page.north east) --
		  (current page.south east) 
		  		    (current page.south west) --
		      (current page.north west) --
		        (current page.north east)}
		]

\maketitle
\tableofcontents

\section{Introduction}
In this paper, we give a de Rham description of the intersection space cohomology theory extending it to a class of Thom-Mather-stratified pseudomanifolds with three strata.

To prove Poincar\'e duality for the resulting cohomology theory, we introduce a proof technique called the method of iterated triangles. Roughly speaking, one proves Poincar\'e duality by induction on the stratification depth using intermediate complexes of forms, distinguished triangles in the derived category over the reals and a Five-Lemma argument in the induction step.
This technique is also applicable for arbitrarily large stratification depth and might be the guideline to generalize the intersection space cohomology theory to a class of Thom-Mather-stratified pseudomanifolds of arbitrary stratification depth.

The theory of intersection spaces, first introduced by M. Banagl in \cite{BanIS}, assigns CW complexes $I^{\pp} X$ to certain types of topological stratified pseudomanifolds. Those depend on a perversity function $\pp$ in the sense of Goresky and MacPherson, see \cite{GM1, GM2}. Their main property is Poincar\'e duality over complementary perversities for the reduced singular (co)homology groups with coefficients in a field. Additionally, using regular singular (co)homology, one gets perversity internal cup products for cohomology.
The construction of the intersection spaces is built upon a homotopy theoretic technique called Moore approximation or spatial homology truncation. The links of the singularities are replaced by a CW-complex (which is not always a subcomplex) with truncated homology. Note, that Moore approximation is an Eckmann-Hilton dual notion of Postnikov approximation. If the singularities are not isolated, one has to perform the Moore approximation equivariantly, see \cite{BanChr}.
Having a perversity internal cup product, intersection space cohomology cannot be isomorphic to intersection cohomology. This is underlined by the behaviour of both theories on cones of smooth manifolds: Intersection (co)homology of a cone equals the truncated (co)homology of the manifold, while intersection space (co)homology of a cone is equal to the cotruncated (co)homology of the manifold.

In \cite{BandR}, Banagl gives a de Rham description of intersection space cohomology, using differential forms on the top stratum of a Thom-Mather-stratified pseudomanifold of stratification depth one with geometrically flat link bundles. A bundle is called flat if the transition functions are locally constant and geometrically flat if, in addition, the structure group of the bundle is contained in the isometries of the fiber. 
Flat link bundles occur in reductive Borel-Serre compactifications of locally symmetric spaces and in foliated stratified spaces. The latter play a role in the work of Farrell and Jones on the topological rigidity of negatively curved manifolds, for instance, see \cite{FJ_FoliatedControlTheory, FJ_CpctNegCurved}.
For such bundles, the Leray-Serre spectral sequence with real coefficients collapses at the $E_2$ page, see \cite[Theorem 5.1]{BanFlatBdles}. 
Examples of flat sphere bundles with nonzero real Euler class, constructed by Milnor in \cite{MilnorCurvzero}, show that one cannot always equip the link of a flat bundle with a Riemannian metric such that the bundle becomes geometrically flat. Banagl uses Riemannian Hodge theory to cotruncate the de Rham complex on the fiber of the link bundle. The geometrical flatness condition then allows to perform that cotruncation fiberwisely. The de Rham complex computing intersection space cohomology consists of all forms on the top stratum of the pseudomanifold with restriction to a collar neighbourhood of the boundary equaling the pullback of a fiberwisely cotruncated form on the boundary. 
 
Banagl establishes a de Rham isomorphism for pseudomanifolds of depth one with only isolated singularities. 
Examples of applications of the intersection space cohomology theory contain K-theory (\cite[Chapter 2.8]{BanIS}), deformation of singular varieties in algebraic geometry (\cite{BM11}), perverse sheaves (\cite{BM14}), geometrically flat bundles and equivariant cohomology (\cite{BanFlatBdles}) and string theory in theoretical physics (\cite[Chapter 3]{BanIS} and \cite{BM14}).

\vspace{2ex}
The purpose of the present paper is a generalization of intersection space cohomology via the de Rham approach to certain pseudomanifolds of stratification depth two. The approach we pursue might be suitable to generalize the theory to pseudomanifolds of arbitrary stratification depth. However, not all the technical difficulties do already arise in the current setting. In \cite{BanD2}, Banagl uses homotopy pushouts of 3-diagrams of spaces to define intersection spaces for first cases of depth two pseudomanifolds. By using the de Rham approach, we enlarge the class of depth two pseudomanifolds intersection space cohomology is applicable to. 
Let $X$ be a compact, oriented, Thom-Mather-stratified pseudomanifold with three strata of different dimension, a zero-dimensional closed stratum (the stratum of maximum codimension), and a geometrically flat link bundle for the intermediate stratum. We then define a subcomplex $\OI$ of the complex of smooth differential forms on the blowup $M$ of $X.$ We prove the following Poincar\'e duality theorem for the cohomology groups $HI_{\pp}^\bullet (X) = H^\bullet \left( \OI (M \right)$.

\vspace{0.2cm}
\noindent
\textbf{Theorem \ref{thm:PD_OI_isolated}:} (Poincar\'e duality for HI) \\
\emph{For all} $r\in \Z,$ \emph{integration induces nondegenerate bilinear forms}
\begin{align*}
\int : HI_{\pp}^r (X) \times HI_{\qq}^{n-r} (X)  \rightarrow \R \\
\bigl ( [ \omega ],  [ \eta ] \bigr )   \mapsto \int_M \omega \wedge \eta.
\end{align*}

\noindent
We give a short overview about the construction of the complex $\OI$ in the depth two case and the idea of the proof of the Poincar\'e duality theorem.
If $X$ has three strata of different dimension, the blowup of the pseudomanifold is a $ \langle 2 \rangle-$manifold $M$.
That is a manifold with corners, such that the boundary decomposes into two smooth manifolds with boundary $E$ and $W$, glued along their common boundary. In our setting, $E$ is the total space of a geometrically flat link bundle, while the connected components of $W$ are (trivially) fibered over points. 
$M$ comes equipped with a system of collars for $E, ~W$ and $\dE = \dW$ (which is induced by the Thom-Mather control data of $X$). The intermediate complex $\wOI (M)$ is defined to contain the smooth forms on $M$ with restriction to the collar neighbourhood of $E$ equaling the pullback of a fiberwisely cotruncated multiplicatively structured form on $E$.  
$\OI (M)$ is then defined to contain the forms of $\wOI (M)$ with restriction to the collar of $W$ equaling the pullback of a cotruncated form on $W$. So, forms in $\OI (M)$ satisfy two different pullback-cotruncation properties. 
Thus, the restriction of the forms to the intersection of the two collar neighbourhoods of the boundary parts \( E \) and \( W \) (which is the collar neighbourhood of $\dE$) has to be both the pullback of an appropriate form on \( E \) as well as the pullback of an appropriate form on \( W \), hence the pullback of some form on \( \dE = \dW \).
This is the main difficulty in the proof of the Poincar\'e duality theorem for the cohomology of \( \OI (M) \).  
As mentioned before, we use the method of iterated triangles to prove Poincar\'e duality. That means that we use a chain of intermediate complexes and prove Poincar\'e-Lefschetz duality statements for them using distinguished triangles and 5-Lemma arguments. This results in Poincar\'e duality for $HI$. In this paper we need only one intermediate complex, namely $\wOI (M)$. For a pseudomanifold of stratification depth $d$ (with flatness conditions on the link bundles) one would need $d-1$ intermediate complexes.

\subsection*{General Notation}
For a smooth $\langle n \rangle-$manifold \( M \) with boundary \( \dM = \partial_1 M \cup ... \cup \partial_p M, \) a collar of the boundary part \( \partial_i M \) is denoted by \( c_i : \partial_i M \times [0,1) \hookrightarrow M \) with \( c_i|_{\partial_i M \times \{0\}} = \id_{\partial_i M} \). We mainly work with \( \langle 2 \rangle \)-manifolds, i.e. manifolds with corners and two boundary parts \( \partial_1 M = E \) and \( \partial_2 M = W\), \(\dM = E \cup_{\dE = \dW} W. \)
The inclusion of the boundary parts is denoted by \( j_E: E \hookrightarrow M \) and \( j_W : W \hookrightarrow M \) and the inclusion of the corner \( \dE = \dW \) by \( j_{\dW} = j_{\dE}: \dW \hookrightarrow M . \)
The image of a collar, \( \im c_i \subset M \) is called a collar neighbourhood.
For a real vector space \( V\), we denote the linear dual \( \text{Hom} (V, \R ) \) by \( V^\dagger \). 
%
%
%
\section*{Acknowledgements}
This paper is based on the dissertation I wrote at the University of Heidelberg under the supervision of M. Banagl. It would not have been possible without his constant support, advice and guidance. For interesting discussions, helpful questions and important remarks, I am also deeply grateful to the members of the topology research group at the University of Heidelberg and to Eugenie Hunsicker from the University of Loughborough in England.
For generous financial support during my time as a PhD student, I thank the Evangelische Studienstiftung.

I am grateful to the anonymous referee for numerous suggestions, which helped improving this paper significantly.

\section{Collars on Bundles and Manifolds with Corners}
\subsection{Width of a collar}
In order to prove Poincar\'e duality for the later defined complexes on manifolds with boundaries we need the following relations between finite 
open covers and collars on manifolds with boundary. Since we only consider compact manifolds in the remainder of the paper, open covers will always have finite subcovers. 

\begin{mydef}[Small collars] \label{def:small_collar}
Let $B$ be a manifold with nonempty boundary $\dB,$ $c: \dB \times [0,1) \hookrightarrow B$ an open collar of the boundary and $\mathcal{U} := \left\{ U_\alpha  \right\}_{\alpha \in I}$ an open cover of $B$. Let $I_\partial \subset I$ be the index set containing the indizes of the $U_\alpha$ with nonempty intersection with the boundary $\dB.$ The collar $c$ is called \emph{small} with respect to the cover $\mathcal{U}$ if the following conditions are satisfied.
\begin{enumerate}
\item   $U_\alpha \cap B_{-} \neq \emptyset,$ for every $ \alpha \in I,$
  where \( B_{-} :=  B -c \bigl ( \dB \times [0,1) \bigr ). \)
\item There exist $W_\alpha \subset \dB$ open with $c \left( W_\alpha \times [0,1) \right ) \subset U_\alpha $ for each $\alpha \in I_\partial$ and such that $ \left\{ W_\alpha \right\}_{\alpha \in I_\partial}$ is an open cover of $\dB.$
\end{enumerate}
\end{mydef}

\begin{lemma}
\label{lemma:widthcollar}
Let \( B \) be a manifold with non-empty boundary \( \dB \), let \( c: \dB \times [0,1) \hookrightarrow B \) be an open collar of \( \dB \) in \( B \) and let \( \mathcal{U} := \{ U_\alpha \}_{\alpha \in I} \) be a finite 
open cover of \( B \) 
. Then there is an \( \epsilon \in (0,1] \) such that the subcollar
\[
 c| : \dB \times [0, \epsilon ) \hookrightarrow B 
\]
is small with respect to $\mathcal{U}$.
\end{lemma}

\begin{proof}
Let \( C = c \bigl ( \dB \times [0,1) \bigr )\).
If there are no \( U_\alpha \in \mathcal{U} \) such that \( U_\alpha \subset C \) we take \( \epsilon = 1 \) and are done. So suppose \( U_\alpha \subset C \). Since \( U_\alpha \subset B \) is open, there must be an \( N_\alpha \in \N \) such that \( U_\alpha \not \subset c \bigl ( B \times [0,1/n ) \bigr ) \) for all \( n \geq N_\alpha \). (Otherwise \( U_\alpha \) would be contained in \( \dB = c ( \dB \times \{0\} ) \).)
Choose such an \( N_\alpha \) for each \( \alpha \in I \) and set \( \epsilon := ( \max_{\alpha \in I} N_\alpha )^{-1} \in (0,1] \). This is well defined since the index set \( I \) is finite and the first relation in the definition is satisfied for that \( \epsilon \).\\
Assume without loss of generality that we could choose $\epsilon = 1$ in the above. Take an $\alpha \in I_\partial$. Let $x \in U_\alpha \cap \dB.$ Since $c^{-1} \left( U_\alpha \cap C \right) \subset \dB \times [0,1) $ is open, there exist $W_x \subset \dB$ open and an $\epsilon_x \in (0,1]$ such that $W_x \times [0,\epsilon_x) \subset c^{-1} \left( U_\alpha \cap C \right)$. If $x$ is contained in more than one of the $U_\alpha$'s, choose $W_x$ and $\epsilon_x$ so small that $c\left( W_x \times [0,\epsilon_x) \right) $ is contained in all these $U_\alpha$'s. Since $\dB$ is compact there are finitely many $x_1, \cdots, x_k \in \dB$ such that the $W_{x_i}$ cover $\dB.$ Let $\epsilon:= \min \epsilon_{x_i}$ and set $W_\alpha := \bigcup_{x_i \in U_\alpha} W_{x_i}.$ By definition, $c \left( W_\alpha \times [0,\epsilon)  \right) \subset U_\alpha$ and $\dB \times [0,\epsilon) \subset \bigcup_{\alpha \in I_\partial} W_\alpha \times [0,\epsilon).$ Hence, the collar $c|: \dB \times [0,\epsilon) \hookrightarrow B$ is small with respect to $\mathcal{U}.$
\end{proof}

\subsection{p-related Collars on Fiber Bundles}
We start with a proposition on \(p\)-related collars on a fiber bundle \(p: E \rightarrow B\)  over a base manifold \(B\) with boundary \( \dB \).
\begin{mydef}(\(p\)-related collars)\\
	Let \( p: E \rightarrow B \) be a smooth fiber bundle with closed smooth fiber \( F \) and \( B \) a compact smooth manifold with boundary \( \dB \). Let 
	\[
		c_{\dE}: \dE \times [0,1) \rightarrow E
	\]
	be a smooth collar on the manifold with boundary \( E \) and
	\[
		c_{\dB}: \dB \times [0,1) \rightarrow B
	\]
	a smooth collar on \( B \). Then \( c_{\dE} \) and  \( c_{\dB} \) are called \(p\)-related if and only if the diagram
	\[ \begin{tikzcd}
			\dE \times [0,1) \ar{r}{c_{\dE}} \ar{d}{p| \times \id} 	& 	E \ar{d}{p} \\
			\dB \times [0,1) \ar{r}{c_{\dB}} 				& 	B
	\end{tikzcd}
\]
commutes.
\end{mydef}

\begin{ex}
Let \( E = L \times B \) be a trivial link bundle. 
We then can take any collar \( c_{\dB}: \dB \times [0,1) \hookrightarrow B \) of \( \dB \) in \( B \) and take \( c_{\dE} := \id_L \times c_{\dB}: \dE \times [0,1) \hookrightarrow E  \). \(c_{\dE} \) is indeed a collar of \( \dE = \dB \times L \), since we work with closed fibers \( L \). Hence the diagram
\[
	\begin{tikzcd}
		\dE \times [0,1) \ar{r}{c_{\dE}} \ar{d}{\pi_2 \times \id} 	& 	 E = L \times B  \ar{d}{\pi_2} \\
		\dB \times [0,1) \ar{r}{c_{\dB}} 				& 	B
	\end{tikzcd}
\]
commutes and the collars are \( p\)-related.
\end{ex}

\begin{prop} \label{prop:prelatedcollars}
	For any smooth fiber bundle \( p : E \rightarrow B \) with base space a compact smooth manifold with boundary \( (B, \dB) \) and closed smooth fiber \( L \) there is a pair of \(p\)-related collars
	\[ \begin{split}
		c_{\dE}:  \dE \times [0,1) \hookrightarrow  E, \\
	c_{\dB}:   \dB \times [0,1) \hookrightarrow B.
	\end{split}
\]
Moreover, if a collar \( c_{\dB} : \dB \times [0,1) \hookrightarrow  B \) is given then a collar \( c_{\dE } : \dE \times [0,1) \hookrightarrow E \) can be chosen such that \( c_{\dE} \) and \( c_{\dB}| \) are \( p\)-related for some subcollar of \( c_{\dB} \). (In detail, one takes a subcollar \( c_{\dB}|_{\dB \times [0,\alpha)} \) for some \( \alpha \in (0,1] \) and reparametrizes it to get a map \( \dB \times [0,1) \hookrightarrow B \).)\\
%
\end{prop}
\begin{proof} We start with the first part and proceed as follows:
\begin{enumerate}
	\item First we construct a vector field \( X \) on \( B \) which is nowhere tangent to \( \dB \). The flow of this vector field then gives the collar \( c_{\dB} \) on \(B\). 
	\item By locally lifting this vector field, we construct a vector field \( Y \) on \( E \) that is nowhere tangent to \( \dE \) and \( p\)-related to \( X \), i.e. for each \( e \in E \) we have
		\[ p_* Y_e = X_{p(e)}. \]
	\item By \cite[Prop 4.2.4]{Abraham}, we then have the relation
		\[ p \circ \eta^Y_t = \eta^X_t \circ p. \]
		for the flows \( \eta^X \) of \( X \) and \( \eta^Y \) of \( Y \).
		That relation implies the statement of the proposition.
\end{enumerate}
The first step is quite simple and standard: Take a finite good open cover \( \bigl \{ U_\alpha \bigr \}_{\alpha \in I} \) of \( B \) such that the bundle trivializes with respect to this cover. Then let \( J \subset I \) denote the set of those \( \alpha \in I \) with \( U_\alpha \cap \dB \neq \emptyset \). For each \( \alpha \in J \) define a vector field \( X_\alpha \) on \( U_\alpha \) by taking the induced vector field of \( \partial_b \) on \( \R^{b-1} \times [0, \infty) \) by the coordinate map \( \nu_\alpha \). Then take a partition of unity \( \{ \rho_\alpha \}_{\alpha \in I} \) subordinate to the cover \( \{ U_\alpha \} \) and define
\[
	X := \sum_{\alpha \in J} \rho_\alpha X_\alpha.
\]
To obtain the vector field \( Y \in \mathfrak{X}(E) = \Gamma (TE) \) we proceed as follows:
Since there is a natural isomorphism between vector bundles
\[
	T (U_\alpha) \times T(L) \stackrel{\cong}{\longrightarrow} T (U_\alpha \times L) 
\]
for all \( \alpha \in I \), we can lift the vector field \( \rho_\alpha X_\alpha \in \mathfrak{X}(U_\alpha) \) to
\(
	(\rho_\alpha X_\alpha,0),
\) a section of \( T(U_\alpha) \times T (L) \cong T(U_\alpha \times L)\), which still has compact support in \( U_\alpha \times L.\)

Since \( p: E \rightarrow B \) is a fiber bundle with fiber \( L\) and \( \bigl \{ U_\alpha \bigr \}_{\alpha \in I} \) a covering of the base \( B \) with respect to which the fiber trivializes, we have a diffeomorphism \( \phi_\alpha : p^{-1} (U_\alpha) \xrightarrow{\cong} U_\alpha \times L, \) for all \( \alpha \in I \), such that the diagram
\[ \begin{tikzcd}
		p^{-1} (U_\alpha) \ar{rr}{\phi_\alpha}[swap]{\cong} \ar{dr}[swap]{p|} & ~ & 
		U_\alpha \times L 
		\ar{dl}{\pi_1} \\
		~ & U_\alpha & ~ 
\end{tikzcd}
\]
commutes.
Note that since the \( \phi_\alpha \) are diffeomorphisms, there exist pushforward vector fields \( {\phi_\alpha^{-1}}_* (\rho_\alpha X_\alpha, 0) \in \mathfrak{X} (p^{-1} (U_\alpha) ) \) with compact support (in \( p^{-1} (U_\alpha) \)). 
Since the family \( \bigl \{ p^{-1} (U_\alpha) \bigr \}_{\alpha \in I} \) is an open cover of \( E \), such that the sets in \( \bigl \{ p^{-1} (U_\alpha) \bigr \}_{\alpha \in J} \) cover an open neighbourhood of the boundary \( \dE \) of \( E\), we can set
\[
	Y := \sum_{\alpha \in J} {\phi_\alpha^{-1}}_* (\rho_\alpha X_\alpha,0) 
\]
to get a vector field \( Y \in \mathfrak{X} (E) \) that is nowhere tangent to \( \dE \). Let \( x \in \dE \), then \( x \in p^{-1} ( U_{\alpha_1...\alpha_r} = U_{\alpha_1} \cap ... \cap U_{\alpha_r}) \) for some \( \alpha_1, ..., \alpha_r \in J \). Then \( Y_x \) is not tangent to \( \dE \) if and only if \( ((\phi_{\alpha_1})_* Y)_{\phi_{\alpha_1} (x)} \) is not tangent to \( \dB \times U_{\alpha_1} \times L \).
\[ \begin{split}
	((\phi_{\alpha_1})_* Y)_{\phi_{\alpha_1} (x)} & = \sum_{i=1}^r \rho_{\alpha_i} (p(x)) \bigl ( [\id \times (\pi_2 \circ \phi_{\alpha_1} \circ \phi_{\alpha_i})]_* (X_{\alpha_i},0) \bigr )_{\phi_{\alpha_1} (x)} \\
	&= \sum_{i=1}^r \rho_{\alpha_i} (p(x)) (X_{\alpha_i})_{p(x)}. 
\end{split} \]
Now this is of course not tangent to the boundary since by definition of the \( X_\alpha \in \mathfrak{X} (U_{\alpha} ) \) we have (with again the \( \nu_\alpha \) the coordinate maps of the base): \( (\nu_{\alpha_1}^{-1})_* X_\alpha = (\nu_{\alpha_1}^{-1} \circ \nu_\alpha)_* \partial_b = \sum_{i=1}^b a_i \partial_i \) with \( a_b > 0 \) since the transition maps are maps between manifolds with boundary.

\vspace{0.2cm}
Further, we have to show that \( X \) and \(Y \) are \(p\)-related, i.e. it holds that
\( p_* Y_e = X_{p(e)} \) for every \( e \in E \). This is equivalent to the statement that for all smooth functions on an open subset of \( B \) it holds that
\[ Y (f \circ p) = (Xf) \circ p.
\]
(see e.g. \cite[Lemma 3.17]{Lee}).
For let \( f : U \rightarrow \R \) be a smooth function on an open subset \( U \subset B \) and let \( x \in p^{-1} (U) \). Then
\[ \begin{split}
		Y(f \circ p) (x) & = Y_x (f \circ p) \\
		& = \sum_{\alpha \in \widetilde{J} }  {\phi_\alpha^{-1}}_* (\rho_\alpha X_\alpha, 0 )_{\phi_\alpha (x)} (f \circ p) 
		\quad \text{with} ~ \widetilde{J} = \{ \alpha \in J | x \in p^{-1} (U_\alpha) \}	\\
		& = \sum_{\alpha \in \widetilde{J}} \bigl (\rho_\alpha  X_{\alpha}, 0 \bigr )_{\phi_\alpha (x)} (f \circ p \circ \phi_\alpha^{-1}) \\
		& = \sum_{\alpha \in \widetilde{J}} \rho_\alpha \bigl ( p(x) \bigr ) \bigl ( (X_\alpha)_{p(x) = \pi_1 \circ \phi_\alpha (x)} , 0_{\pi_2 \circ \phi_\alpha (x)} \bigr ) (f \circ \pi_1 ) \\
		& = \sum_{\alpha \in \widetilde{J}} \rho_\alpha \bigl ( p(x) \bigr ) (X_\alpha)_{p(x)} (f) = X_{p(x)} (f) = (Xf) \bigl (p(x) \bigr ).
\end{split}
\]

\vspace{0.2cm}
As mentioned, for every \( t \), this implies the relation
\begin{equation}
	p \circ \eta_t^Y = \eta^X_t \circ p 
	\label{eq:bdleflows}
\end{equation}
for the flows \( \eta^X \) of the vector field \( X \in \mathfrak{X}(B) \) and \( \eta^Y \) of \( Y \in \mathfrak{X} (E) \) (which are embeddings by \cite[Theorem 9.2.4]{Lee}). This relation implies the claim since there are open neighbourhoods \( W_B \subset \dB \times [0, \infty) \) of \( \dB \) and \(W_E \subset \dE \times [0, \infty) \) of \( \dE \) respectively, such that the flows \( \eta^X \) and \( \eta^Y \) are defined on these open subsets. But then there are constants \( \epsilon_B, \epsilon_E > 0 \) such that
\( \dB \times [0, \epsilon_B ) \subset W_B \)
and
\( \dE \times [0, \epsilon_E ) \subset W_E. \)
Let \( \epsilon := \min (\epsilon_B, \epsilon_E ) \) and let \( f : [0,1) \rightarrow [0, \epsilon)  \) be a diffeomorphism. Then we have collar embeddings
\[ c_{\dB}: \dB \times [0,1) \stackrel{\id \times f}{\longrightarrow} \dB \times [0, \epsilon) \stackrel{\eta_X}{\longrightarrow} B \]
and
\[ \begin{tikzcd} 
		c_{\dE}: \dE \times [0,1) \ar{r}{\id \times f} & \dE \times [0, \epsilon) \ar{r}{\eta_Y} &  E  
\end{tikzcd}
\]
such that
\[\begin{split}
		p \circ c_{\dE} ~ (x,t) & = (p \circ \eta^Y_{f(t)}) (x)  = \eta^X_{f(t)} \bigl ( p(x) \bigr ) \quad \bigl ( \text{by eq.} ~ (\ref{eq:bdleflows}) \bigr )\\
		& = c_{\dB} \circ (p \times \id)  (x,t).
\end{split}
\]

\vspace{0.5cm}
For the second part of the proof we proceed likewisely, but take a special vector field in step 1: The collar allows us to define a vector field \( \widetilde{X} \in \mathfrak{X} (C_{\dB} ) \) (with \( C_{\dB} = \im c_{\dB} \)) by taking the pushforward of \( \partial_t \): \( \widetilde{X} = (c_{\dB})_* \partial_t  \). Then for any \( q \in \dB \) and any \( f \in C^{\infty} (C_{\dB}) \) it holds that
\[ (\widetilde{X}f) \bigl ( c_{\dB} (\tau,q) \bigr ) = \bigl ( \partial_t (f \circ c_{\dB}) \bigr ) (\tau,q) 
	= \frac{d}{dt} \Bigl |_{t=\tau} (f \circ c_{\dB}) (t,q). \]
This means that the flow of the vector field restricted to the boundary \( \dB \) is the given collar \( c_{\dB} \). 
We then "lift" this vector field as before, not to a vector field on the whole total space \( E \) but rather to a vector field \( Y \in \mathfrak{X} \bigl ( p^{-1} (C_{\dB}) \bigr ) \), where \( p^{-1} (C_{\dB}) \) is an open neighbourhood of the boundary, by setting
\[ Y = \sum_{\alpha \in J } {\phi_\alpha^{-1}}_* (\rho_{\alpha} \widetilde{X}|, 0 ). \]
As before this defines a nowhere vanishing vector field which is nowhere tangent to the boundary \( \dE \). The rest is a complete analogy to the first step. Note that it suffices to have the vector fields on open neighbourhoods of the boundary since we later only need the flow of the vector fields restricted to the boundary.
\end{proof}

\begin{rem}[$p$-related collars and local trivializations]\label{rem:prelColTriv}
Note that for a small collar $c_{\dB}$, as in Definition \ref{def:small_collar}, the construction of $c_{\dE}$ in the proof of previous proposition , the maps $\pi_2 \circ \phi_\beta \circ c_{\dE}|_{p^{-1} (W_\beta)},$ with $\phi_\beta: p^{-1} (U_\beta) \to U_\beta \times L$ the local trivializations, $\pi_2: U_\alpha \times L \to L$ the second factor projection and the $W_\beta \subset \dB$ belonging to the small collar, are independent of the collar coordinate if the bundle is flat. This is true, since the restriction of the vector field $Y = \sum_{\alpha \in J} {\phi_\alpha^{-1}}_* (\rho_\alpha X_\alpha,0)$ to some $p^{-1}\left( c_{\dB}(W_\beta \times [0,1) ) \right), ~ \beta \in J$ is
\begin{align*}
Y|_{p^{-1} \left( c_{\dB} (W_\beta \times [0,1) ) \right) } 	&= \sum_{\alpha \in J} ( \phi_\beta^{-1} \circ \phi_\beta )_* \; (\phi_\alpha^{-1})_{*} \; (\rho_\alpha X_\alpha, 0) \\
			&= (\phi_\beta^{-1})_* \sum_{\alpha \in J} (\id \times g_{\beta \alpha})_* \; (\rho_\alpha X_\alpha, 0) = (\phi_\beta^{-1})_* \sum_{\alpha \in J} (\rho_\alpha X_\alpha,0 ).
\end{align*}
The statement of \cite[Prop 4.2.4]{Abraham} then gives the relation 
\[ \eta^Y_t = \phi_\beta^{-1} \circ \, (\eta'_t \times \id_L) \circ \phi_\beta \]
for the flows $\eta^Y_t$ of $Y$ and $\eta'_t$ of $\sum_{\alpha \in J} \rho_\alpha X_\alpha.$ The following calculation shows, that this gives the statement. (Recall the notation from the proof of the previous proposition).
\[ \pi_2 \circ \phi_\beta \circ c_{\dE} (x,t) = \pi_2 \circ \phi_\beta \circ \phi_\beta^{-1} \circ \left( \eta'_{f(t)} \times \id_L \right) (\phi_\beta (x))  = \pi_3 \circ \phi_\beta (x), \]
with $\pi_3: W \times [0,1) \times L \to L$ the projection. 
From now on, we assume that all the $p$-related collars with $c_{\dB}$ small on flat bundles in usage satisfy this relation. We will use this, e.g. in the proof of Lemma \ref{lemma:MSdc}, to move integration in the collar direction of multiplicative forms near the boundary $\dE$ to the base factor.
\end{rem}

\subsection{Collars on Manifolds with Corners} \label{subsection:CollarsMfdsCorners}
We are going to work with differential forms on a smooth manifold with corners \( M^n \), the boundary of which can be subdivided as \( \dM = E \cup_{\dE = \dW} W \), satisfying certain conditions near the boundary parts \( E \) and \( W \). In order to define ``near \(E, ~ W \)'' precisely we have to investigate how the concept of a collar on a manifold with boundary generalizes to manifolds with corners of that type.
In \cite{VeronaStratMaps}, the author proves a theorem, see \cite[Theorem 6.5]{VeronaStratMaps}, that can be interpreted as a transition between Thom-Mather-stratified pseudomanifolds and manifolds with faces: Any such pseudomanifold can be obtained from a manifold with faces by making certain identifications on the faces.

\begin{mydef}(Manifolds with Faces)\\
	Let \( M^n \) be an \( n \)-dimensional manifold with corners and for each \( x \in M \) let \( c(x) \) denote the number of zeroes of \( \phi (x) \in \R^n_+ = [0, \infty)^n\) for any coordinate chart \( \phi : U \rightarrow \R^n_+ \) with \( x \in U \). A face is the closure of a connected component of the set \( \{ p \in M | c(p) = 1 \} \). Then \( M \) is called a manifold with faces if each \( x \in M \) is contained in \( c(x) \) different faces. 
\end{mydef}

\begin{ex}\ \\
\begin{minipage}[l]{0.49\textwidth}
A \( 2 \)-dimensional disc with one corner is a manifold with corners but not with faces, since the corner point does not lie in \( 2 \) faces but only in one.
\end{minipage}
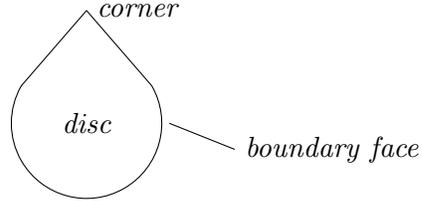
\begin{minipage}[r]{0.49\textwidth}
\begin{center}
\begin{tikzpicture}
\draw[
    draw,
    line join=round]
 (30:1) arc (30:-210:1)
to[rounded corners=0] (90:1.5)
to cycle;
\node at (0,0) {disc};
\node[right] at (90:1.5) {corner};
\draw (0:1.1) -- (-10:2);
\node[below,right] at (-10:2) {boundary face};
    \end{tikzpicture}
\end{center}
\end{minipage}
\end{ex}

As mentioned, manifolds with faces are considered by Verona in \cite[Chapter 4]{VeronaStratMaps} to examine triangulability of stratified mappings, but also by Alpert in \cite[Section 3]{Alpert} to estimate simplicial volume, by J\"anich in \cite{Jaenich} to classify $O(n)-$Manifolds and by Laures in \cite{Laures} to investigate cobordisms on manifolds with corners.

The latter two authors also define \( \langle n \rangle \)-manifolds, which are manifolds with faces together with a decomposition of the boundary into $n$ faces that satisfy the following relations. 

\begin{mydef}(\( \langle n \rangle \)-manifolds) [See \cite[Def. 1]{Jaenich}] \\
	A manifold with faces \( M \) together with a \( n \)-tuple of faces \( (\partial_0 M, ... , \partial_{n-1} M ) \) is called an \( \langle n \rangle \)-manifold if 
	\begin{enumerate}
		\item \( \dM = \bigcup_{i=0}^{n-1} \partial_i M \),
		\item \( \partial_i M \cap \partial_j M \) is a face of both \( \partial_i M \) and \( \partial_j M \) if \( i \neq j \).
	\end{enumerate}
\end{mydef}

Note that a \( \langle 0 \rangle \)-manifold is just a usual manifold (without boundary) and a \( \langle 1 \rangle \)-manifold is a manifold with boundary. Simple examples of \( \langle n \rangle \)-manifolds for arbitrary \( n \in \N \) are  \( \R_+^n \) or the standard \( n \)-simplex. We focus on \( n= 2 \).

\vspace{0.2cm}
So let \( M^n \) be an \( n \)-dimensional \( \langle 2 \rangle \)-manifold with faces \( E, ~ W \) (hence \( \dE = \dW \)).
By \cite[Lemma 2.1.6]{Laures} there are collars 
\[
	\begin{split}
		c_{\dE}: ~& \dE \times [0,1) \hookrightarrow E, \\
		c_{\dW}: ~& \dW \times [0,1) \hookrightarrow W, \\
		c_E: ~& E \times [0,1) \hookrightarrow M ~ \text{and} \\
		c_W: ~& W \times [0,1) \hookrightarrow M,\\
	\end{split}
\]
with \( C_X:= c_X( X \times [0,1) ) \) for \( X = \dE, ~ \dW, ~ E, ~ W \) such that \[ c_E|_{\dE \times [0,1)} = c_{\dW} \] and \[ c_W|_{ \dW \times [0,1)} = c_{\dE}. \]

\begin{prop}
	Let \( M^n \) be a \( \langle 2 \rangle\)-manifold with boundary \( \dM = E \cup W \) as before. Then any two collars \( c_{\dE}: \dE \times [0,1)  \hookrightarrow E \) and \( c_{\dW}: \dW \times [0,1) \hookrightarrow W \) extend to collars
	\[ \begin{split}
			c_E : ~& E \times [0,1) \hookrightarrow M, \\
			c_W : ~& W \times [0,1) \hookrightarrow M,
		\end{split} \]
	i.e. \( c_E|_{\dW \times [0,1)} = c_{\dW} \) and \( c_W|_{\dE \times [0,1)} = c_{\dE} \).
\end{prop}
\begin{proof} The proof is simple: Interpret the collars as flows of vector fields on \( E, ~ W \) which do not vanish on the boundaries and point inwards and extend them to vector fields on \( M \) (for example using an arbitrary collar on \( M \)) which do not vanish anywhere on \( W, ~ E \) and point away from $W,~E$, respectively. The flows of these vector fields are collars \( c_W \) and \( c_E \) with the desired properties.
\end{proof}

\begin{cor} \label{cor:exnicecollars_isolated}
	As before, let \( M \) be a \( \langle 2 \rangle \)-manifold with boundary \( \dM = E \cup_{\dE} W \). Assume furthermore that \( E \) is the total space of a fiber bundle \( p: E \rightarrow B \) with closed fiber \( L \) and a compact base manifold with boundary \( B \). Then there are collars \( c_E, ~ c_W \) of \( E, ~ W \) in \( M \) and \( c_{\dB} \) of \( \dB \) in \( B \) such that \( c_W|_{\dE \times [0,1)}  \) and \( c_{\dB} \) are \( p\)-related.
\end{cor}
\begin{proof} Take a pair of \( p\)-related collars of \( \dE \) in \( E \) and of \( \dB \) in \( B \) and any collar of \( \dW \) in \( W \) and then use the previous Proposition to extend them to collars of \( E, ~ W \) in \( M \).
\end{proof}

\section{Thom-Mather-stratified Pseudomanifolds with depth 2}
\subsection{Thom-Mather-stratified Spaces}\label{subs:ThomMather}
If one wants to work with differential forms there has to be some smooth structure. Hence we do not work with topological stratified pseudomanifold as defined for example in \cite[Definition 4.1.1]{BanTI} but use Thom-Mather smooth stratified spaces. We use the definition of B. Hughes and S. Weinberger, cf. \cite[sect. 1.2]{WeinbergerSurgery}. (Another older reference is \cite{MatherTopStability}.)
In this paper we work with \( C^\infty \)-Thom-Mather stratified pseudomanifolds. In \cite{MatherTopStability} and \cite{MatherStratMappings}, Mather proved that every Whitney stratified space has a \( C^\infty \)-Thom-Mather stratification. Since Whitney showed in \cite{WhitneyTangents} that any complex or real analytic set admits a Whitney stratification, those are examples for the type of spaces we consider.

Note further, that Mather also proved, using Thom's isotopy lemmas, that any stratum \( X_i \) in a Thom-Mather stratified space has a neighbourhood \( N \) such that the pair \( (N, X_i) \) is homeomorphic to the pair \( (\text{cyl} (f), X_i) \), with \( \text{cyl} (f) \) the mapping cylinder of some fiber bundle \( p: E \rightarrow X_i \), which is called the link bundle of the stratum. We will later assume these bundles to satisfy flatness conditions.

As we already mentioned in Section \ref{subsection:CollarsMfdsCorners}, Verona observes in \cite{VeronaStratMaps} that Thom-Mather-stratified pseudomanifolds are closely related to manifolds with faces. The Thom-Mather control data corresponds to a system of collars on these manifolds.

As a side remark, we also allude that, by a theorem of Goresky (see \cite{GoreskyTriang}), each \( C^{\infty} \)-Thom-Mather stratified pseudomanifold can be (smoothly) triangulated by a triangulation compatible with the filtration and hence is a PL-pseudomanifold.

\subsection{Flat Fiber Bundles }  \label{subs:bdleprop}
We recall the definition of geometrically flat fiber bundles.
\begin{mydef}((Geometrically) Flat fiber bundles) \label{def:GeomFlat}

	A fiber bundle \( p: E \rightarrow B \) of smooth manifolds with fiber \( L \)  is called flat if there is an atlas \( \mathcal{U} := \{ U_\alpha \}_{\alpha \in I} \) of the bundle such that the corresponding transition functions are locally constant. That means that for the local trivialization maps 
	\( \phi_{\alpha} : p^{-1} (U_\alpha) \stackrel{\cong}{\longrightarrow} U_\alpha \times L \), \( \pi_1 \circ \phi_\alpha = p \), it holds that
	\[ \phi_{\beta} \circ \phi_{\alpha}^{-1} = \id \times g_{\alpha \beta} : (U_{\alpha} \cap U_{\beta}) \times L \rightarrow (U_\alpha \cap U_{\beta}) \times L \]
	with \( g_{\alpha \beta} \in \text{Diffeo} ~ (L) \) if \( U_\alpha \cap U_\beta \) is connected.\\
	A fiber bundle is called geometrically flat if it is flat and if there is a Riemannian metric on the fiber such that the structure group of the bundle is the isometriy group of the link with respect to that metric, i.e. the \( g_{\alpha \beta} \) in the above definition are isometries of \( L \).
\end{mydef}
Note that if the base $B$ is a smooth manifold with boundary, then the same holds for the total space $E$ and the restriction of the bundle to the boundary $ p|: \dE \rightarrow \dB$ is also a fiber bundle with the same flatness properties as the original bundle $p.$

\subsection{The Depth One Setting}\label{subs:DepthOne}
In \cite{BandR}, Banagl investigates
oriented, compact smooth Thom-Mather stratified pseudomanifolds with filtration
\[
	X = X_n \supset X_b = \Sigma
\]
with \( \Sigma^b \) a \( b \)-dimensional connected closed manifold with geometrically flat link bundle. 
That means there is an open neighbourhood \( N \) of \( \Sigma \) in \( X \), such that the boundary of the compact manifold \( M = X - N \) is the total space of a geometrically flat link bundle \( p : \del M \rightarrow \Sigma \) with fiber an oriented, closed smooth 
Riemannian manifold \( L^m \) of dimension \( m = n-1-b \). 
There are two strata in this setting: \( X_b = \Sigma \) and \( X_n - X_b \).

Banagl defines a complex of differential forms \(\OI\) on the nonsingular part \( M \) of \( X \)
using cotruncation in the fiber direction for multiplicatively structured forms on the boundary \( \del M \).\\
The flat link bundle condition allows us to define a complex of multiplicatively structured differential forms on the boundary. Let therefore \( \mathcal{U} := \{U_\alpha\}_{\alpha \in I} \) be a good open cover of \( \Sigma \) such that the bundle trivializes with respect to this cover, 
i.e. for each \( \alpha \in I \) there are diffeomorphisms \( \phi_\alpha : U_\alpha \times L \rightarrow p^{-1} (U_\alpha) \) such that the following diagram commutes:
\[ \begin{tikzcd}
		p^{-1} (U_\alpha) \ar{dr}[swap]{p|} & ~ & 
		U_\alpha \times L \ar{ll}{\cong}[swap]{\phi_\alpha}
		\ar{dl}{\pi_1} \\
		~ & U_\alpha & ~ 
\end{tikzcd}
\]
We are then able to define the following subcomplex of the complex \( \Omega^\bullet (\dM) \) of differential forms on \( \dM \), using the projections \( \pi_1: U_\alpha \times L \rightarrow U_\alpha \) and \( \pi_2: U_\alpha \times L \rightarrow L \):
\[ \begin{split}
	\Omega_{\MS}^\bullet (\Sigma) := \Bigl \{ \omega \in \Omega^\bullet(\del M ) \Bigl | ~ 
	&	\omega|_{p^{-1} (U_\alpha)} = \phi_\alpha^* \sum_j \pi_1^* \eta_j \wedge \pi_2^* \gamma_j  \\
	~&	\text{with} ~ \eta_j \in \Omega^\bullet (U_\alpha), \gamma_j \in \Omega^\bullet (L) \Bigr \}.
\end{split}
\]
These forms can be truncated or cotruncated in the link direction (see \cite[section5]{BandR}) and the mentioned complex \( \OI \) is defined as containing the forms that look like the pullback of a fiberwisely cotruncated multiplicative structured form near \( \del M \) in a collar neighbourhood of the boundary.

The cohomology of that complex then satisfies generalized Poincar\'e-duality over complementary perversities and is isomorphic to the cohomology of the associated intersection space if the link bundle is trivial. For arbitrary (geometrically) flat link bundle, we do not yet know how to construct the intersection space, so there is no cohomology theory coming from a space, we can compare the differential form approach to.

\subsection{Spaces of Stratification Depth Two} \label{subs:3strata}
The aim of the paper is to generalize the above construction to certain classes of pseudomanifolds with stratification depth two. Strictly speaking, we consider smooth Thom-Mather stratified pseudomanifolds \( X \) of dimension \( n \) with filtration \( X = X_n \supset X_b \supset X_s \) with \( n-2 \geq b > s \) and additional conditions on the regular neighbourhoods of the singular strata. The strata here are \( X_s \) and \( X_b - X_s \), which are the singular strata, and \( X_n - X_b \).
We mainly consider zero dimensional bottom strata, i.e. \( s = 0\) and \( X_s = \{ x_0, ... , x_d \} \). 

We consider Thom-Mather-stratified pseudomanifolds \(X\) with filtration
\[
	X=X_n \supset X_b \supset X_0 = \{ x_0, ... , x_d \},
\]
where the bottom stratum is zero dimensional and the middle stratum satisfies a geometrical flatness condition.
To define intersection space cohomology on these stratified pseudomanifolds we first remove a regular neighbourhood \( R_0 \) of \( X_0 \) homeomorphic to \( \mathring{\cone} (L_0) \), with \( L_0 \) a stratified pseudomanifold of dimension \( n-1 \). The result is a stratified pseudomanifold \( X' = X - R_0 \) with boundary and one singular stratum
\[
		B := X'_b = X_b - R_0 \cap X_b,
	\]
	a \( b-\)dimensional compact smooth manifold with boundary \( \dB \).
	We assume that this singular stratum has a geometrically flat link bundle in \( X' \), i.e. there is an open tubular neighbourhood \( T_b\) of \( B \) in \( X'\) such that
	\[
			M := X' - T_b
		\]
		is a smooth \(\langle 2 \rangle \)-manifold with boundary decomposed as
		\[
				\dM = E \cup_{\dE = \dW} W
			\]
such that $E$ is the total space of a geometrically flat link bundle
			\[
					\begin{tikzcd}
								L \ar{r} & E \ar{d}{p} \\
										~ & B
											\end{tikzcd}
										\]
										over the compact base manifold with boundary \( B \) with link a closed smooth Riemannian manifold \( L^m \). We call this $\langle 2 \rangle$-manifold $M$ the blowup of $X$. Note, that it is diffeomorphic as a $\langle 2 \rangle$-manifold to the blowup of \cite{SignaturePackageWitt}, although the authors use a different construction there.
The manifold with boundary \( W \) is 
										\[
												W = L_0 - T_b \cap L_0,
											\]
											with boundary \( \dW = \dE \). It can be seen as the blowup of the pseudomanifold \( L_0 = \partial X' \), the link of \( X_0 \).

\begin{rem}[Thom-Mather control data and Collars]
The control data of the Thom-Mather stratification of a pseudomanifold induces a system of collars on the $\langle 2 \rangle-$manifold $M$. In particular, also a collar of $\dB$ on $B$ is induced such that we get compatible collars on the fiber bundle, as was explained in Corollary \ref{cor:exnicecollars_isolated}. We always work with this system of collars.
\end{rem}

\subsection{Cotruncation Values}\label{subs:CotrValues}
If we have a stratified pseudomanifold \( X \) with filtration \( X = X_n \supset X_b \supset X_0 \) and complementary perversities \( \pp, ~ \qq \), then, unless otherwisely stated, we set \( \dim (L = \text{Link} ~ X_b ) := m := n -1 - b \) and define the cutoff values
\[ \begin{split} K & := m - \pp (m+1) , ~ K' := m - \qq (m+1) ~ \text{and} \\
		L & := n-1 - \pp (n) , ~ L' := n - 1 - \qq (n).
\end{split}
\]
Note, that the dimension of the link $L_0$ of $X_0$ is $\dim L_0 = n-1.$
These cutoff values are the cotruncation degrees for the complexes of multiplicatively structured differential forms near the respective strata.

\section{The method of iterated triangles}
To prove the main Theorem \ref{thm:PD_OI_isolated}, we iterate 5-Lemma arguments involving diagrams containing long exact sequences that are induced by distinguished triangles or short exact sequences. In the setting of this paper we need the intermediate complex $\wOI$, see Section \ref{section:wOI}. The two distinguished triangles \ref{lemma:disttri_tilde_1}, \ref{lemma:disttri_tilde_2} relate the absolute and relative $\wOI$-complexes to the absolute and relative complexes of differential forms on the blowup $M$ of $X$ and complexes of fiberwisely (co)truncated multiplicatively structured forms on the boundary part $E$. A 5-Lemma argument proves Poincar\'e-Lefschetz duality for the cohomology of $\wOI$, using the Poincar\'e-Lefschetz statements for forms on $M$ and the fiberwisely (co)truncated multiplicatively structured forms on $E$. The short exact sequences of Lemma \ref{lemma:disttriOI1} then relate $\OI$ to the absolute and relative $\wOI-$complexes and the complexes of (co)truncated $\OI$-forms on the boundary part $W$, which can be seen as the blowup of the link of $X_0$. The Poincar\'e duality for $HI$ can then be deduced with a 5-Lemma argument, using the Poincar\'e-Lefschetz duality of $\wOI$ and the Poincar\'e duality of $HI_{\pp}^\bullet(W)$, which follows from \cite[Theorem 8.2]{BandR}.

Note that we use the triangle and 5-Lemma argument twice, where two is also the stratification depth of $X$ and hence the number of boundary parts of the blowup. It might be possible, that an analogous construction might help to prove Poincar\'e duality for $HI$ for pseudomanifolds of greater stratification depth with one pair of distinguished triangles for each singular stratum. In analogy to our setting, these triangles would relate a chain of intermediate complexes and certain complexes on the boundary parts of the respective singular strata.

One problem of pseudomanifolds of greater stratification depth (or more difficult bottom stratum) does not occur in this paper: If the fiber of a link bundle is a pseudomanifold itself, and one wants to work with multiplicatively structured forms on the total space of that bundle, one must cotruncate the complex of $\OI$-forms on this fiber such that the resulting complex is compatible with the transistion functions. Since the bottom stratum is assumed to be zero dimensional in this paper, we can use any cotruncation of the $\OI$ complex of the link of the bottom stratum.

\section{The Partial de Rham Intersection Complex 
} \label{section:wOI}
We define the intermediate complexes
\(
	\wOI (M)	
\)
and \( \wOI (M, C_W) \). They consist of forms whose restriction to \( C_E \) is the pullback of a fiberwisely cotruncated form on \( E \) and whose restriction to \( C_W \) is either the pullback of some form on \( W \) or zero for the relative group. We show that the corresponding cohomology groups \( H^r \bigl ( \wOI (M) \bigr ) \) and \( H^{n-r} \bigl ( \wOIq (M, C_W) \bigr ) \) are Poincar\'e-Lefschetz dual to each other, see Theorem \ref{prop:PDonwOI}.

Not till then we define the actual complex of intersection space forms on \(M\), \( \OI \), and show Poincar\'e duality for it.

\vspace{0.2cm}
Before we give the definitions of \( \wOI (M) \) and \( \wOI (M,C_W) \) we recall the definitions of the complex of multiplicatively structured forms as well as the complex of fiberwisely truncated and cotruncated multiplicatively structured forms from \cite[Sections 3 and 6]{BandR}:

\begin{mydef}[Multiplicatively structured forms] \label{def:ms}
	Let \( p: E \rightarrow B \) be a flat bundle with base \( B \) a compact manifold with boundary \( \dB \) and fiber a Riemannian manifold \( L \) and let \( \mathcal{U} = \{ U_{\alpha} \}_{\alpha \in I} \) be a good open cover of \( B \) such that the bundle trivializes with respect to that cover. Let further \( U \subset B \) be open. We then define
	\begin{equation} \label{eq:OMS} \begin{split}
			\OMS (U)  := \bigl \{ \omega \in \Omega^\bullet (p^{-1} (U)) \bigl | & \forall \alpha \in I: \omega|_{p^{-1} (U_\alpha)}  = \phi_{\alpha}^* \sum_{j_\alpha} \pi_1^* \eta_{j_\alpha} \wedge \pi_2^* \gamma_{j_\alpha} \\
				& \text{with} ~ \eta_{j_{\alpha}} \in \Omega^\bullet (U \cap U_{\alpha} ), ~ \gamma_{j_\alpha} \in \Omega^\bullet (L) \bigr \} 
	\end{split}
	\end{equation}
	Here, the \( \phi_\alpha : p^{-1} (U_\alpha) \xrightarrow{\cong} U_\alpha \times L \) denote the local trivializations of the bundle.
\end{mydef}

To define the complexes of fiberwisely truncated and cotruncated multiplicatively structured forms we need the complexes of truncated and cotruncated forms of the closed (Riemannian) manifold \( L \) from \cite[Section 4]{BandR}.
\begin{mydef}[Fiberwisely (co)truncated forms] \label{def:fiberwisetrun}
	Let \( p: E \rightarrow B \) be a flat bundle with base \( B \) a compact manifold with boundary \( \dB \) and fiber a closed manifold \( L \) and let \( \mathcal{U} = \{ U_{\alpha} \}_{\alpha \in I} \) be a good open cover of \( B \) such that the bundle trivializes with respect to that cover as in the previous definition. Let further \( U \subset B \) be open. We then define, for any integer \( K \), the complex of (in degree \( K \)) fiberwisely truncated multiplicatively structured forms by
	\[ \begin{split}
			\ftsOMS (U)  := \bigl \{ \omega \in \OMS (U)  
				\bigl | & \forall \alpha \in I: 
				\omega|_{p^{-1} (U_\alpha)}  = ~ \phi_{\alpha}^* \sum_{j_\alpha} \pi_1^* \eta_{j_\alpha} \wedge \pi_2^* \gamma_{j_\alpha} \\
				& \text{with} ~ 
				\gamma_{j_\alpha} \in \tau_{<K} \Omega^\bullet (L) \bigr \} 
	\end{split}
\]
If the fiber is a Riemannian manifold and the bundle is geometrically flat, we moreover define the complex of fiberwisely cotruncated multiplicatively structured forms by
	\[ \begin{split}
			\ftgOMS (U) := \bigl \{ \omega \in \OMS (U)  
				\bigl | & \forall \alpha \in I: 
				\omega|_{p^{-1} (U_\alpha)}  = ~ \phi_{\alpha}^* \sum_{j_\alpha} \pi_1^* \eta_{j_\alpha} \wedge \pi_2^* \gamma_{j_\alpha} \\
				& \text{with} ~  \gamma_{j_\alpha} \in \tau_{\geq K} \Omega^\bullet (L) \bigr \}.
	\end{split}
\]

\end{mydef}

All of these complexes \( \OMS (U), \ftsOMS (U),\) and \( \ftgOMS (U) \), for \( U \subset B \) open, are subcomplexes of the complex of forms \( \Omega^\bullet \bigl  (p^{-1} (U) \bigr ) \). Note, that we need the geometrical flatness condition to define the complexes $\ftgOMS (U)$ but not $\ftsOMS (U)$, since for any smooth map $f: L \to L$ it holds that the pullback of a boundary is a boundary, $f^* (d\alpha) = d (f^* \alpha)~$ for all $\alpha \in \Omega^\bullet (L)$, but the pullback of a coclosed form is not always coclosed, although this holds if $f$ is an isometry. See \cite[Lemma 4.5]{BandR} for more details.

\vspace{1ex}
Before we define the intermediate intersection form complex $\wOI (M),$ we recall the notation of the collars on the $\langle 2 \rangle$-manifold $M$. The boundary $\dM$ of $M$ consists of the two compact manifolds $E, \, W$ with common boundary $\dE = \dW$. Each of this boundary parts comes with a collar map
\begin{align*}
c_E: E \times [0,1) & \hookrightarrow M, \\
c_W: W \times [0,1) & \hookrightarrow M
\end{align*}
such that $c_E|_{\dE}$ is a collar of $\dW$ in $W$ and $c_{W}|_{\dW}$ is a collar of $\dE$ in $E$. We denote the images of these collars by the capital C's, $C_E = \im c_E$ and $C_W = \im c_W$ and call them the collar neighbourhoods of the respective boundary parts. Moreover, we denote by \( \pi_E, ~ \pi_W \) the projections $ \pi_E :  E \times [0,1) \rightarrow  E,$ $\pi_W: W \times [0,1) \rightarrow W. $
See Section \ref{subsection:CollarsMfdsCorners} for more details.

\begin{mydef}[The Partial de Rham Intersection Complex] \label{mydef:parOI}
Let $K := m - \pp (m), ~ m = \dim (L)$ be the middle cotruncation value, as defined in Section \ref{subs:CotrValues}. We then define the intermediate intersection form complex as follows.
	        \[
		\begin{split}
	                \widetilde{\Omega I}^\bullet_{\pp} (M) := \bigl \{ \omega \in \Omega^\bullet (M) \bigl | 
			      &	c_E^* \omega = \pi_E^* \eta  \text{ for some} ~ \eta \in ft_{\geq K} \Omega^\bullet_{\MS} (B)  \\
			      & c_W^* \omega = \pi_W^* \rho \text{ for some} ~  \rho \in \Omega^\bullet (W) \bigr \} \\
		\end{split}
		\]
This is a subcomplex of the complex \( \Omega^\bullet (M) \) of forms on \( M \).
\end{mydef}

\begin{mydef}[The relative Partial de Rham Intersection Complex] \label{mydef:relparOI}
	 \[
	 \widetilde{ \Omega I }^\bullet_{\pp} (M,C_W) := \bigl \{ \omega \in \widetilde{ \Omega I}^\bullet_{\pp} (M) \bigl | c_W^* \omega = 0 \bigr \}	\subset \Omega^\bullet (M, C_W).
	 \]
\end{mydef}
In the rest of this section we prove Poincar\'e duality between \( \widetilde{ \Omega I}^\bullet_{\pp} (M) \) and \( \widetilde{ \Omega I}^\bullet_{\qq} (M,C_W) \) for complementary perversities $\pp$ and $\qq$. Recall that two perversities are called perversities, if $\pp (k) + \qq (k) = k-2 $ for all $k.$

\vspace{0.2cm}
\noindent
\textbf{Theorem \ref{prop:PDonwOI}} \emph{(Poincar\'e Duality for Partial Intersection Forms)} \\
\emph{For complementary perversities \( \pp \) and \( \qq \), integration induces a nondegenerate bilinear form}
	\[
	\begin{split}
		\widetilde{HI}^r_{\pp}(M) \times \widetilde{HI}^{n-r}_{\qq} (M,C_W)  &~ \rightarrow \R \\
		([\omega], [\eta] ) & ~ \mapsto \int_M \omega \wedge \eta ,\\
	\end{split}
\]
\emph{where} \(\widetilde{HI}^r_{\pp}(M) := H^r(\widetilde{\Omega I}^\bullet_{\pp} (M)) \) \emph{and} \(  \widetilde{HI}^{n-r}_{\qq} (M,C_W) := H^{n-r} (\widetilde{ \Omega I}^\bullet_{\qq} (M,C_W) )  \).

\subsection{Multiplicative Forms that are constant near the end}
The proof of the Theorem \ref{prop:PDonwOI} is geared on the proof of Poincar\'e duality for intersection forms in the two strata case, see \cite[Sect. 8]{BandR}. 
However, the additional stratum produces additional technical difficulties, as one might have expected.
We first deal with the fact that in \(\widetilde{\Omega I}_{\pp}^\bullet \) we do not just demand that the forms restricted to a collar neighbourhood of \(E\) come from a form in \( ft_{\geq K} \Omega_{\MS}^\bullet (B) \) but also that they are constant in the collar direction in a collar neighbourhood of \( W \).
\begin{mydef}[Fiberwise cotruncated forms that are in \( \Omega_{\partial c}^\bullet (E) \)]\label{def:Pcomplexes}
	We recall that the collar \( c_{\partial E}: \dE \times [0,1) \hookrightarrow E \) of \( \dE \) in \( E \) is the restriction of the collar of \( W \) in \( M \), \( c_{\dE} = c_W|_{\dW \times [0,1)} \), and  define 
	\[
		P^\bullet (B) := \{ \omega \in \Omega_{\MS}^\bullet (B) | ~ \exists~ \eta \in \Omega_{\MS}^\bullet (\dB): ~c_{\dE}^* \omega = \pi_{\dE}^* \eta \}.
	\]
Analogously, we define
\[
	P_{\geq K}^\bullet (B) := \{ \omega \in ft_{\geq K}\Omega_{\MS}^\bullet (B) |~ \exists ~\eta \in ft_{\geq K}\Omega_{\MS}^\bullet (\dB):~ c_{\dE}^* \omega = \pi_{\dE}^* \eta  \}
	\]
and
\[
	P_{<K}^\bullet (B) := \{ \omega \in ft_{<K}\Omega_{\MS}^\bullet (B) |~ \exists ~ \eta \in ft_{<K} \Omega_{\MS}^\bullet (\dB): ~ c_{\dE}^* \omega = \pi_{\dE}^* \eta  \}.
	\]
\end{mydef}
We want to show that those complexes are quasi-isomorphic to the analogous complexes without the condition at the end of the manifold. 
The following lemma shows, that argument of the proof of \cite[Prop. 2.4]{BandR} is applicable to multiplicatively structured forms as well.
The proof there uses integration of the forms on the collar,
$\rho: \Omega^\bullet (E) \rightarrow \Omega_{\partial C}^\bullet (E)$,
with $\Omega^\bullet_{\partial C} = \left\{ \omega \in \Omega^\bullet (E) | \exists \eta \in \Omega^\bullet (\dE):~ c_{\dE}^* \omega = \pi_{\dE}^* \eta \right\}.$ We recall the exact definition in the proof below.

\begin{lemma}
	\label{lemma:MSdc}
	The subcomplex inclusions \( i: P^\bullet (B) \hookrightarrow \Omega_{\MS}^\bullet(B) \), \( i_{\geq K}: P_{\geq K}^\bullet (B) \hookrightarrow \ftgOMS (B) \), \( i_{<K}: P_{<K}^\bullet (B) \hookrightarrow \ftsOMS (B) \) are quasi-isomorphisms.
\end{lemma}
\begin{proof} 
We give a proof for the non-truncated case that transfers literally to the truncated and cotruncated one.
Take a slight extension of the $p-$related collars $c_{\dE}$ and $c_{\dB}$ to a pair $\widetilde{c}_{\dE} : \dE \times [0, \frac{3}{2}) \hookrightarrow E, ~\widetilde{c}_{\dB} : \dB \times [0, \frac{3}{2}) \hookrightarrow B$ of collars which are still $p-$related and such that $\widetilde{c}_{\dB}$ is still a small collar. Let $\widetilde{\pi}_{\dE}: \dE \times [0, \frac{3}{2}) \to \dE$ and $\widetilde{\pi}_{\dB}: \dB \times [0, \frac{3}{2}) \to \dB$ denote the corresponding projections on the first factors and 
$\text{proj}_2: \dB \times [0, \frac{3}{2}) \to [0, \frac{3}{2})$ the second factor projection.

Define a smooth cutoff function $ \xi: [0, \frac{3}{2} ) \rightarrow \R$ with compact support and $\xi|_{[0,1]} = 1.$ This induces a cutoff function $p^* (\widetilde{c}_{\dB}^{-1})^* \text{proj}_2^* \xi $ on $E$,
which we also denote by $\xi$. Trivially, this is a multiplicative function. Let $\omega \in \Omega^\bullet (E)$ be any form. Then $\widetilde{c}_{\dE}^* \omega$ decomposes as $\omega_0 + dt \wedge \omega_1,$ with $\omega_0 (t), \omega_1 (t) \in \Omega^\bullet (\dE)$ for each $t \in [0, \frac{3}{2} ).$ Let $a \in (0,1)$ and define a map $\rho: \Omega^\bullet (E) \rightarrow \Omega_{\partial C}^\bullet (E)$, 
where the latter is the subcomplex of forms $\omega$ with $c_{\dE}^* \omega $ constant in the collar direction, by
\[ \rho (\omega) = (1 - \xi) \omega + \xi ({\widetilde{c}_{\dE}}^{-1})^* \pi_1^* \omega_0 (a) - d \xi \wedge (\widetilde{c}_{\dE}^{-1})^* \int_a^t \omega_1 \, dt. \]
By the argument of Banagl, this map is a chain homotopy equivalence with homotopy inverse the subcomplex inclusion  $\Omega_{\partial C}^\bullet (E) \hookrightarrow \Omega^\bullet (E).$ The homotopy is given by
\( K ( \omega ) = \xi \,(\widetilde{c}_{\dE}^{-1})^* \int_a^t \omega_1 \, dt. \) 
We prove that $\rho$ and $K$ restrict to complexes of multiplicatively structured forms in the following.

By our choice of $\xi,$ this can be achieved by proving that for $\omega_1$ with $(\widetilde{c}_{\dE}^{-1})^* \omega_1$ multiplicative in the above decomposition, integration yields a multiplicative form $ (\widetilde{c}_{\dE}^{-1})^* \int_a^t \omega_1 \, dt.$
So let $\omega_1 = c_{\dE}^* \widetilde{\omega}_1,$ with $\widetilde{\omega}_1$ a multiplicatively structured form. 
Recall, that we work with a collar that is small with respect to the (finite) open cover $\mathcal{U} = \left\{ U_\alpha \right\}_{\alpha \in I}$ with respect to which the bundle trivializes, see Definition \ref{def:small_collar} and Lemma \ref{lemma:widthcollar}. Let $I_\partial $ and the $W_{\alpha} $ be as in the aforesaid definition. Let $\left\{ \rho_\alpha \right\}_{\alpha \in I_\partial}$ be a partition of unity on $\dB$ with respect to the cover $\left\{ W_\alpha \right\}_{\alpha \in I_\partial}.$ This gives a partition 
of unity $ \left\{ \overline{\rho}_\alpha:= (\widetilde{c}_{\dB}^{-1})^* \widetilde{\pi}_{\dB}^* \rho_\alpha \right\}_{\alpha \in I_\partial}$ of the collar of $\dB$ in $B.$  
Since $\widetilde{c}_{\dB} \left( W_\alpha \times [0, \frac{3}{2}) \right) \subset U_\alpha$ and $\text{supp} (\rho_\alpha) \subset W_\alpha,$ the support of $\overline{\rho}_\alpha$ is contained in $U_\alpha$.
Hence $\omega_1 = \sum_{\alpha} p^* \overline{\rho}_\alpha \omega_1.$ By assumption, we can write $\omega_1$ as follows.
\[ \omega_1 = \widetilde{c}_{\dE}^{*} \sum_{\alpha \in I_\partial} \phi_\alpha^* \sum_{j_\alpha} \pi_1^* (\overline{\rho}_\alpha \eta_{j_\alpha} ) \wedge \pi_2^* \gamma_{j_\alpha}. \]
This can be reformulated by Remark \ref{rem:prelColTriv}, since we use $p-$related collars $\widetilde{c}_{\dE}$ and $\widetilde{c}_{\dB}$ with $c_{\dB}$ small.
\[ \omega_1 = \sum_{\alpha \in I_\partial} \sum_{j_\alpha} \underbrace{\widetilde{c}_{\dE}^{*}~ \phi_\alpha^*~ \pi_1^*}_{= \widetilde{c}_{\dE}^*~ p^* = (p \times \id)^* ~ \widetilde{c}_{\dB}^*} (\overline{\rho}_\alpha \eta_{j_\alpha} ) \wedge \underbrace{\widetilde{c}_{\dE}^{*}~ \phi_\alpha^*~ \pi_2^*~ \gamma_{j_\alpha}}_{\text{independent of}~t}. \]
Hence, if we integrate $\omega_1 \in \Omega^\bullet \left( \dE \times [0,1) \right)$ along the collar coordinate, we can move integration to the first factor of the wedge product.
\[ \int_a^t \omega_1 (\tau) \, d\tau = \sum_{\alpha \in I_\partial} \sum_{j_\alpha}  (p \times \id)^* \left( \widetilde{\pi}_{\dB}^* \rho_\alpha \int_a^t (\widetilde{c}_{\dB}^* ~\eta_{j_\alpha}) (\tau) \, d\tau \right) \wedge \widetilde{c}_{\dE}^{*}~ \phi_\alpha^* \pi_2^* \gamma_{j_\alpha}.\]
The pullback of this form under $\widetilde{c}_{\dE}^{-1}$ is indeed multiplicatively structured:
For any $\beta \in I,$ let $\phi_\beta: p^{-1} (U_\beta) \to U_\beta \times L$ be the local trivialization of the bundle. We then have the following commutative diagram, where the $W_\beta$ are as in Definition \ref{def:small_collar}.
\[
\begin{tikzcd}
p^{-1} (W_\beta) \times [0,\frac{3}{2}) \ar[hookrightarrow]{r}{\widetilde{c}_{\dE}} \ar{d}{p \times \id} & p^{-1} (U_\beta) \ar[shift left=0.5ex]{rr}{\phi_\beta} \ar{dr}[swap]{p} & \ &  U_\beta \times L \ar[shift left=0.5ex]{ll}{\phi_\beta^{-1}} \ar{dl}{\pi_1} \\
W_\beta \times [0,\frac{3}{2}) \ar[hookrightarrow]{rr}{\widetilde{c}_{\dB}} & \ & U_\beta & \
\end{tikzcd}
\]
This gives the relation $(p \times \id) \circ \widetilde{c}_{\dE}^{-1} \circ \phi_\beta^{-1} = \widetilde{c}_{\dB}^{-1} \circ p \circ \phi_\beta^{-1} = \widetilde{c}_{\dB}^{-1} \pi_1$.
Now, we restrict $(\widetilde{c}_{\dE}^{-1})^* \int_a^t \omega_1 (\tau) \, d\tau $ to $p^{-1} (U_\beta)$ and use the transition functions $\rho_{\alpha \beta} = \id \times g_{\alpha \beta}: (U_\alpha \cap U_\beta) \times L \to (U_\alpha \cap U_\beta) \times L$ of the flat bundle. 
\begin{align*}
& ( \widetilde{c}_{\dE}^{-1})^* \int_a^t \omega_1 (\tau) \, d\tau ~\Bigr|_{p^{-1} (U_\beta)} 
= \phi_\beta^*  (\phi_\beta^{-1})^* (\widetilde{c}_{\dE}^{-1})^* \int_a^t \omega_1 (\tau) \, d\tau \\
= & ~\phi_\beta^* \sum_{\alpha \in I_\partial} \sum_{j_\alpha} \overbrace{ (\phi_\beta^{-1})^* (\widetilde{c}_{\dE}^{-1})^* (p \times \id)^*}^{= \pi_1^* (\widetilde{c}_{\dB}^{-1})^*} \left( \widetilde{\pi}_{\dB}^* \rho_\alpha \int_a^t (\widetilde{c}_{\dB}^* ~\eta_{j_\alpha}) (\tau) \, d\tau \right) \\
& \hspace{5ex} \wedge \underbrace{(\phi_\beta^{-1})^* ~ \phi_\alpha^* \pi_2^*}_{= \rho_{\alpha \beta}^* \pi_2^* = \pi_2^* g_{\alpha \beta}^*} \gamma_{j_\alpha}
= \phi_\beta^* \sum \pi_1^* (\ldots) \wedge \pi_2^* (\ldots).
\end{align*}
This shows that $(\widetilde{c}_{\dE}^{-1})^* \int_a^t \omega_1 (\tau) \, d\tau$ is multiplicatively structured and hence $\rho: \OMS (B) \to P^\bullet (B)$ is a chain homotopy equivalence with inverse the subcomplex inclusion $i: P^\bullet (B) \hookrightarrow \OMS (B).$
\end{proof}
\subsection{Two Distinguished Triangles for \texorpdfstring{$ \wOI $}{partial Omega I}}
To prove Theorem \ref{prop:PDonwOI}, we use a five lemma argument and therefore need two distinguished triangles in \( \mathcal{D} ( \R ) \), the derived category over the reals.
\begin{mydef}[Forms that are multiplicative near E] \label{def:EMS}
	\[
		\Omega_{E \MS}^r (M) := \bigl \{ \omega \in \Omega^r (M) \bigl | ~ \exists ~ \eta \in \Omega^\bullet_{\MS}( B ): ~ c_E^* \omega = \pi_E^* \eta \bigr \}.
	\]
\end{mydef}

\begin{lemma} \label{lemma:disttri_tilde_1}
	In \( \mathcal{D} ( \R ) \), the derived category of complexes of real vector spaces, there is a distinguished triangle 
\begin{equation}
	\label{disttri_1}
	\widetilde{\Omega I}_{\pp}^\bullet (M) 
	\rightarrow
	\Omega_{E \MS}^\bullet (M)
	\rightarrow
	ft_{<K} \Omega_{\MS}^\bullet (B) 
	\rightarrow \wOI (M)[+1] 
\end{equation}
\end{lemma}
\begin{proof} There is a short exact sequence 
\begin{equation}
	\label{eqn:ses_dtr}
	0 \rightarrow \widetilde{\Omega I}_{\pp}^\bullet (M) \rightarrow \Omega_{E \MS}^\bullet (M) \rightarrow Q_E^\bullet(M) := \frac{\Omega_{E \MS}^\bullet (M)}{\widetilde{\Omega I}_{\pp}^\bullet (M) } \rightarrow 0
\end{equation}
We have to show that there is a quasi-isomorphism \( Q_E^\bullet (M) \rightarrow ft_{<K} \Omega_{\MS}^\bullet (B) \). Let \( \sigma_E: E \hookrightarrow E \times [0,1) \) be the inclusion at \( 0 \). 
Then the map \( J_E := c_E \circ \sigma_E \) induces maps
\[
	\begin{split}
		J_E^* &: 	\Omega_{E \MS}^\bullet (M) 
		\stackrel{c_E^*}{\twoheadrightarrow}
		\Omega_{E \MS}^\bullet (E \times [0,1)) \xrightarrow[\cong]{\sigma_E^*} \Omega_{\MS}^\bullet (B) \\
		\widetilde{J_E}^* &:  \widetilde{\Omega I}_{\pp}^\bullet (M) \stackrel{c_E^*}{\twoheadrightarrow} \widetilde{\Omega I}^\bullet (E \times [0,1)) \xrightarrow[\cong]{\sigma_E^*} P_{\geq K}^\bullet (B) \\
		\overline{J_E}^* &: Q_E^\bullet (M) \stackrel{\overline{c_E}^*}{\twoheadrightarrow} Q_{E}^\bullet (E \times [0,1)) \xrightarrow[\cong]{\overline{\sigma_E}^*} Q^\bullet (B) := \frac{\Omega_{\MS}^\bullet (B)}{P_{\geq K}^\bullet (B) }.  \\
	\end{split}
\]
The induced maps \( J_E^* \) and \( \widetilde{J_E}^* \) are surjective by the standard argument of enlarging the collar and using a bump function. 
Their kernel is $\ker (J_E^*) = \ker (\widetilde{J_E}^*) = \Omega^\bullet (M, C_E),$ consisting of forms that vanish on the collar neighbourhood $C_E$ of $E$ in $M.$ A $3\times 3-$Lemma then implies that the map \( \overline{J_E}^*: Q_E^\bullet (M) \rightarrow Q^\bullet (B) \) is an isomorphism. 
By Lemma \ref{lemma:MSdc}, subcomplex inclusion induces a quasi-isomorphism
\[ \overline{i}: Q^\bullet (B) \stackrel{qis}{\longrightarrow} \frac{\OMS (B)}{\ftgOMS (B) } \]
and since we work with a flat bundle \( E \) over \( B \), there is a quasi-isomorphism
\begin{equation} \label{eq:gammaB}
	\gamma_B: ft_{<K} \Omega_{\MS}^\bullet (B) \rightarrow \frac{\OMS (B)}{\ftgOMS (B) }
\end{equation}
by \cite[Lemma 6.7]{BandR}. All in all we get a fraction of quasi-isomorphisms
\[ \begin{tikzcd}
		~ & \frac{\OMS (B)}{\ftgOMS (B) } & ~ \\
		Q_E (M) \arrow{ru}{\overline{i} \circ J_E^*}[swap]{qis} & ~ & \ftsOMS (B) \ar{lu}{qis}[swap]{\gamma_B} 
\end{tikzcd}
\]
in the derived category \( \mathcal{D} (\R) \) which allows us to replace \( Q_E^\bullet (M) \) in (\ref{eqn:ses_dtr}) to get the desired distinguished triangle in \( \mathcal{D} ( \R ) \). 
\end{proof} 

\begin{mydef}[Relative de Rham complexes]\label{def:relderham}
\[
	\begin{split}
		\Omega_{rel}^\bullet (M) & := \{ \omega \in \Omega^\bullet (M) | c_E^*\omega = 0, c_W^* \omega =0 \} \\
		\widetilde{\Omega I}_{\pp}^\bullet (M,C_W) & := \{ \omega \in \widetilde{\Omega I}_{\pp}^\bullet (M) | c_W^* \omega = 0 \} \\
		ft_{\geq K} \Omega_{\MS}^\bullet (B,C_{\dB}) & := \{ \omega \in ft_{\geq K} \Omega_{\MS}^\bullet (B) | c_{\dE}^* \omega = 0 \} .
	\end{split}
\]
\end{mydef}

\begin{rem}
Note, that by the our choice of collars $c_{\dE}, \, c_{\dB},$ see Remark \ref{rem:prelColTriv}, the last of the previously described complexes can be characterized by moving the vanishing condition on the collar neighbourhood to the base parts of the multiplicative forms:
\begin{align*}
& \omega \in \ftgOMS (B, C_{\dB}) \\
\Leftrightarrow  ~ 
&\forall \, U \in \mathfrak{U}: ~ \omega|_{p^{-1} (U)} = \phi_U^* \sum_j \pi_1^* \eta_j \wedge \pi_2^* \gamma_j \quad \text{with} \\
&\eta_j \in \Omega^\bullet (U, U \cap C_{\dB}),\, \gamma_j \in \tau_{\geq K} \Omega^\bullet (L).
\end{align*}
To see this, let $U \in \mathfrak{U}$ and let $W \subset \dB$ be an open subset such that $c_{\dB} (W \times [0,1) ) \subset U.$ Since the collar $c_{\dB}$ is small (recall Definition \ref{def:small_collar}), the collar neighbourhood $C_{\dB} \subset B$ can be covered by finetely many such sets $c_{\dB} (W \times [0,1) ).$ Since we work with $p$-related collars, the collar neighbourhood $C_{\dE} \subset E$ can be covered by sets $c_{\dE} (p^{-1} (W) \times [0,1) ) = p^{-1} \left( c_{dB} (W \times [0,1) ) \right)$, for finetely many such $W \subset \dB$. Let $c_{\dB}|$ be the restriction of $c_{\dB}$ to $W \times [0,1)$ and let $c_{\dE}|$ be the restriction of $c_{\dE}$ to $p^{-1} (W) \times [0,1)$. We then get the following for $\omega \in \ftgOMS (B, C_{\dB})$, using the results of Remark \ref{rem:prelColTriv}.
\begin{align*}
0 = &(c_{\dE}|)^* (\omega|_{p^{-1} (U)}) = (c_{\dE}|)^* \phi_U^* \sum_j \pi_1^* \eta_j \wedge \pi_2^* \gamma_j \\
= & \sum_j (p \times \id)^* (c_{\dB}|)^* \eta_j \wedge \phi_U^* \pi_3^* \gamma_j
=  (\widetilde{\phi}_U)^* \sum_j \pi_1^* (c_{\dB}|)^* \eta_j \wedge \pi_2^* \gamma_j,
\end{align*}
where $\widetilde{\phi}_U: p^{-1} (W) \times [0,1) \to W \times [0,1) \times L$ is the trivialization related to $p \times \id_{[0,1)}$.
This implies $(c_{\dB}|)^* \eta_j = 0$ and hence $c_{\dB}^* \eta_j = 0$
, because the sets $c_{\dB} ( W \times [0,1))$ cover $C_{\dB} \subset B.$ To be more precise, assume without loss of generality that $W$ is a coordinate chart and let
\[ 
0 = \sum_j \pi_1^* \eta_j \wedge \pi_2^* \gamma_j = \sum_{I} \sum_{j=1}^{k_I} f_{j}^I dx^I \wedge \gamma_{j}^I \in \Omega^\bullet (W \times [0,1) \times L), \]
		where we sum over all multi-indizes \( I \) of coordinates on $W \times [0,1)$. The advantage of this notation is, that we sorted the forms in the first factor by the different $dx^I$ an can treat each multi-index \( I \) seperately. Assume that there is an \( j_0 \in \{1,...,k_I \} \) and an \( x \in W \times [0,1) \) such that \( f_{j_0}^I (x) \neq 0. \) Contracting with \( \partial_x^I \) and evaluating at \( x \), this gives:
		\[ \gamma_{j_0}^I = - \sum_{j \neq j_0} \frac{f_{j}^I (x)}{f_{j_0}^I (x)} \gamma_{j}^I. \]
		Therefore we can write
		\[ \sum_{j=1}^{k_I} f_{j}^I dx^I \wedge \gamma_{j}^I = \sum_{j \neq j_0} \bigl ( f_{j}^I - \underbrace{\frac{f_{j}^I (x)}{f_{j_0}^I (x)}}_{=: c_{j}^I} f_{j_0}^I \bigr ) dx^I \wedge \gamma_{j}^I. \]
		If these new coefficient functions \( f_{j}^I - c_{j}^I f_{j_0}^I \) vanish on \( W \times [0,1) \) we are done. Otherwise, repeat this process inductively to reduce the above sum to just one summand \( \widetilde{f}^I dx^I \wedge \gamma^I \), for some \( \gamma^I \), which still must equal the sum we started with and is thereby zero. Then either \( \gamma^I = 0 \) or \( \widetilde{f}^I = 0. \)
		\end{rem}

\begin{rem}
	The cohomology groups of the above defined complexes do not depend on the choice of a pair of \(p\)-related collars. This can be deduced by a spectral sequence argument, see \cite[Section 11]{PhD}.
\end{rem}

\begin{lemma} \label{lemma:disttri_tilde_2}
	There is a second distinguished triangle
\begin{equation}
	\label{disttri_2}
	\Omega_{rel}^\bullet (M) 
	\rightarrow
	\widetilde{\Omega I}_{\pp}^\bullet (M,C_W) 
	\rightarrow
	ft_{\geq K} \Omega_{\MS}^\bullet (B, C_{\dB}) 
	\rightarrow 	\Omega_{rel}^\bullet (M)[+1]
\end{equation}
\end{lemma}
\begin{proof} The map \( \widetilde{J_E}^*: \widetilde{\Omega I}_{\pp}^\bullet (M, C_W) \rightarrow ft_{\geq K} \Omega_{\MS}^\bullet (B, C_{\dB}) \) is surjective by the standard argument of enlarging the collar and using a bump function. The kernel of \( \widetilde{J_E}^* \) are those forms \( \omega \in \widetilde{\Omega I}_{\pp}^\bullet (M,C_W) \) with
$ c_E^* \omega = 0 $
and hence \( \ker \widetilde{J_E}^* = \Omega_{rel}^\bullet (M).\) Therefore, we have a short exact sequence
\[
	0 \rightarrow \Omega_{rel}^\bullet (M) \longrightarrow \widetilde{\Omega I}_{\pp}^\bullet (M,C_W) \stackrel{\widetilde{J_E}^*}{\longrightarrow} ft_{\geq K} \Omega_{\MS}^\bullet (B, C_{\dB}) \rightarrow 0
\]
which induces the Distinguished Triangle (\ref{disttri_2}) in the derived category.

\end{proof}
\subsection{Poincar\'e Duality for Fiberwisely (Co)truncated Forms}
\begin{prop}
\label{prop:PD_ftMS}
	Let $K$ and $K^*$ be the cutoff values for complementary perversities $\pp$ and $\qq$ defined in Section \ref{subs:CotrValues}. Then, for any \( r \in \mathbb{Z} \), integration induces a nondegenerate bilinear form 
\[
	\begin{gathered}
		\int: H^{r-1} (ft_{<K} \Omega_{\MS}^\bullet (B)) \times H^{n-r}(ft_{\geq K^*} \Omega_{\MS}^\bullet (B,C_{\dB}))  \rightarrow \R , \\
		([\omega],[\eta])  \mapsto \int_E \omega \wedge \eta.
	\end{gathered}
\]

\end{prop}
For being able to prove the above Proposition \ref{prop:PD_ftMS}, we need two Poincar\'e Lemmata and a Bootstrap Principle:

\begin{lemma} (Poincar\'e Lemma for fiberwisely truncated forms) \label{lemma:PLfibtr} \\
	Let \( U \subset B \) be a coordinate chart intersecting the boundary $\dB$, that means the bundle \( p: E \rightarrow B \) trivializes over \( U \). In detail, there is a diffeomorphism \( \phi_U: p^{-1} (U) \stackrel{\cong}{\longrightarrow} U \times L \) with \( p = \pi_1 \circ \phi_U \). Let further denote \( \pi_2: U \times L \rightarrow L \) the second factor projection and \( S_x: L \stackrel{at ~x}{\longrightarrow} U \times L \) the inclusion at \( x \in U - (\dB \cap U) \). Then the induced maps
\[
\begin{tikzcd}
	ft_{<K} \Omega_{\MS}^\bullet (U) \arrow[shift left=0.5ex]{r} {S_x^*}  & \tau_{<K} \Omega^\bullet (L) \arrow[shift left=0.5ex]{l}{(\pi_2 \circ \phi_U)^*}
\end{tikzcd}	
\]
are chain homotopy inverses of each other. In particular both are homotopy equivalences.
\end{lemma}
\begin{proof}
The proof is an analogy to the proof of \cite[Lemma 5.1]{BandR}. The only difference is that for charts intersecting the boundary (which means that $U$ is diffeomorphic to the half space $\R^n_+ = [0,\infty) \times \R^{n-1}$ and not to $\R^n$), $\R^0 \hookrightarrow \R_+ = [0, \infty)$ is embedded at $1 \in [0, \infty).$ This does not change the argument, though. 
\end{proof}

\begin{mydef}
	For any open subset \( U \subset B\) we define
	\[ \begin{split}
		\Omega_{\MS}^\bullet (U, U \cap C_{\dB}) &:= \bigl \{ \omega \in \Omega_{\MS}^\bullet (U) \Bigl | \omega|_{\underbrace{p^{-1}(U \cap C_{\dB})}_{= p^{-1} (U) \cap C_{\dE}}} = 0 \bigr \}, \\
		\OMSc (U, U \cap C_{\dB} ) & := \bigl \{ \omega \in \OMSc (U) \Bigl | \omega|_{\underbrace{p^{-1}(U \cap C_{\dB})}_{= p^{-1} (U) \cap C_{\dE}}} = 0 \bigr \}
	\end{split}
	\]
	Analogously, we define the fiberwisely truncated and cotruncated subcomplexes.
\end{mydef}

In the following lemma we give the induction start for the Mayer--Vietoris argument, which makes use of the fact that the collar we work with is small with respect to the chosen good open cover \( \mathcal{U} \) (compare to \ref{lemma:widthcollar}).

\begin{lemma}(Poincar\'e Lemma for relative forms with compact supports) \label{lemma:PLrc}
	Let \( U \in \mathcal{U} \) be an open chart (with respect to which the bundle trivializes, i.e. there is a diffeomorphism \( \phi_U : p^{-1} (U) \rightarrow U \times L \) with \( p|_{p^{-1} (U)} = \pi_1 \circ \phi_U \)). Then in particular there is a diffeomorphism \( \psi : U \stackrel{\cong}{\rightarrow} V \) with \( V = \R^n_+ \) or \( V = \R^n \) and, by Lemma \ref{lemma:widthcollar}, \( U \) is \underline{not} completely contained in the collar neighbourhood \( C_{\dB} \supset \dB \) of the boundary of \( B \). 
	Then there is a form \( e \in \Omega_c^n( U, U \cap C_{\dB} ) = \{ \omega \in \Omega_c^n (U) ~|~ \omega|_{U \cap C_{\dB}} = 0 \} \) such that the maps 
	\[
	\begin{tikzcd}
		ft_{\geq K} \Omega_{\MS,c}^\bullet (U, U \cap C_{\dB}) \arrow[shift left=0.5ex]{rr}{(\pi_2)_* \circ (\phi_U^{-1})^*} &  & \tau_{\geq K} \Omega^{\bullet-n} (L) \arrow[shift left=0.5ex]{ll}{e_*},
\end{tikzcd}	
\]
where 
\[
	{\pi_2}_* (\pi_1^* \eta \wedge \pi_2^* \gamma ) = 
	\begin{cases} 	\bigl ( \int_{U} \eta \bigr ) ~ \gamma 	& \text{if} ~ \eta \in \Omega^n_c (U, U \cap C_{\dB}),\\
		0 						& \text{else},
	\end{cases}
\]
and
\begin{equation} \label{eq:def_e} 
       	e_* (\gamma) = \phi_U^* (\pi_1^* e \wedge \pi_2^* \gamma), 
\end{equation}
are chain homotopy inverses of each other and in particular are both chain homotopy equivalences.
\end{lemma}
\begin{proof}
\emph{First step:} (Definition of the form \( e \))

Independent of \( U \) being diffeomorphic to \( \R^n \) or \( \R^n_+ \) we can assume that \( \psi (U) = V \subset \R^n \) is arranged in such a way that for, say the \( x^0 \) component of elements \( x \in V \) large enough, \( x^0 > s \), one has \( x \not \in  \psi (C_{\dB} \cap U ) \) (for \( V = \R^n_+ \), \( x^0 \) is also a component such that \( \partial \R^n_+ = \{ x^0 = 0 \} \)).
We then take bump functions \( \epsilon_i \in C_0^\infty ( \R )  \) with \( \int_\R \epsilon_i = 1 \) for \( i \in \{ 0, ... , n-1 \} \), such that in addition \( \text{supp} (\epsilon_0) \subset (s, \infty ) \). But then 
\[ e := \psi^* ~ \Bigl ( \prod_{i=0}^{n-1} \epsilon_i \Bigr ) ~ dx^0 \wedge ... \wedge dx^{n-1} \in \Omega_c^n ( U, U \cap C_{\dB} ). \]
	The map \( e_*: \tau_{\geq K} \Omega^r (L) \rightarrow ft_{\geq K} \Omega_{\MS,c}^{r+n} (U, U \cap C_{\dB} ) \) is defined by relation (\ref{eq:def_e}) and by the definition of the form \( e \) it holds that \( (\pi_2)_* \circ (\phi_U^{-1})^* \circ e_* = \id \).

\emph{Second step:} (Construction of the homotopy operator)
As in the proof of \cite[Lemma 5.5]{BandR} and in the proof of the previous Lemma \ref{lemma:PLfibtr}, we prove by induction on \( n\) that \( e_* \circ (\pi_2)_* \circ (\phi_U^{-1})^* \simeq \id \). 
Note, that the complex $ft_{\geq K } \Omega_{\MS, c}^{\bullet -1} (\R^{n-1} )$ is defined on \cite[p.21]{BandR}, where Banagl uses the symbol $F$ instead of $L$.
First, let $e_0 = \epsilon_0 dx^0 \in \Omega^ ([0,\infty)).$ We show that the maps \[ 
\begin{split}
{e_0}_* : ft_{\geq K } \Omega_{\MS, c}^{\bullet -1} (\R^{n-1} ) & \longrightarrow  \ftgOMSc ( U, U \cap C_{\dB} ) \\
{e_0}_* ( \pi_1^* \eta \wedge \pi_2^* \gamma) &:= \phi_U^* \bigl ( \pi_1^* \psi^* ( e_0 \wedge \pi^* \eta) \wedge \pi_2^* \gamma )
\end{split} \]
		with \( \pi : V \rightarrow \R^{n-1} \) the projection, and 
		\[ \pi_*: \ftgOMS (U , U \cap C_{\dB}) \rightarrow ft_{\geq K} \Omega_{\MS, c}^{\bullet -1} (\R^{n-1}) \] 
		(integration along the first fiber coordinate) defined by 
		\begin{align*}
				\pi_* \bigl ( \phi_U^* ( \pi_1^* \psi^* ( \underbrace{f(t,x) du^J}_{\text{no}~ dt ~ \text{contained} } ) \wedge \pi_2^* \gamma ) \bigr ) & = 0 \\
				\pi_* \bigl ( \phi_U^* ( \pi_1^* \psi^* (g(t,x) dt \wedge du^J ) \wedge \pi_2^* \gamma ) \bigr ) &= \pi_1^* \int_{\R} g(t,x) dt ~ du^J \wedge \pi_2^* \gamma 
		\end{align*}
	satisfy the relation \( {e_0}_* \circ \pi_* \simeq \id \) and hence are mutually inverse homotopy equivalences.
	The homotopy operator \[ K: \ftgOMSc (U, U \cap C_{\dB} ) \rightarrow ft_{\geq K } \Omega_{\MS, c}^{\bullet -1} (U, U \cap C_{\dB} ) \] satisfying \( dK + Kd = {e_0}_* \circ \pi_* \) is defined by	
\[ \begin{split}
		& K  \bigl ( \phi_U^* ( \pi_1^* \psi^* (f(t,x) du^J ) \wedge \pi_2^* \gamma ) \bigr )   = 0 \\
		& K \bigl ( \phi_U^* ( \pi_1^* \psi^* (g(t,x) dt \wedge du^J) \wedge  \pi_2^* \gamma )  \bigr ) \\
		& = \phi_U^* ( \pi_1^* \psi^* ( \int_{-\infty}^t g(\tau,x) d\tau - \int_{- \infty}^t e_0 ) ~ \int_{\R} g(\tau, x) d \tau ) du^J \wedge \pi_2^* \gamma ).
\end{split}
\]
\emph{Remark: For $U \cong \R^n_+$ the lower integration limits in the definition of $K$ and $\pi_*$ must be changed from $-\infty$ to $0.$}\\
By our definition of \( e_0 \), \( K \) respects the vanishing condition. A standard calculation shows that \( Kd + dK = {e_0}_* \circ \pi_* - \id \).

The second step is to put together the first step with the result of \cite[Lemma 5.5]{BandR}: The following diagram commutes
		\[ \begin{tikzcd}
				\ftgOMSc (U, U \cap C_{\dB}) \ar{r}{\pi_*} \ar[shift left=1ex, bend left=20]{rr}{{\pi_2}_* \circ (\phi_U^{-1})^*} &
				ft_{\geq K} \Omega_{\MS,c}^{\bullet -1} (\R^{n-1})  \ar[shift left=1ex]{l}{{e_0}_*} \ar{r}{\widetilde{\pi}_*} &
				\tau_{\geq K} \Omega^{\bullet -n} (L) \ar[shift left=1ex]{l}{\widetilde{e}_*} \ar[shift left=1ex, bend left=20]{ll}{e_*}
		\end{tikzcd}
	\]
	Note that \( \widetilde{e}_* \) and \( \widetilde{\pi}_* \) denote the mutually inverse homotopy equivalences of \cite[Lemma 5.5]{BandR}. The commutativity of this diagram then implies the statement of the lemma: Since \( e_* = \widetilde{e}_* \circ {e_0}_* \) and \( {\pi_2}_* = \pi_* \circ \widetilde{\pi}_* \) are the composition of mutually inverse homotopy equivalences, they are also mutually inverse homotopy equivalences.
\end{proof}

To use a Mayer--Vietoris type argument we need a bootstrap principle. The following lemma will provide one in our case:
\begin{lemma} (Bootstrap principle) \label{lemma:bootsmapMS} \\
	Let \( U, V \subset B \) be open sets and let \( b:= \dim B \), \( m = \dim L \). Then if
	\[ \begin{gathered}
	\int : ~  H^r \bigl (\ftsOMS (Y) \bigr ) \times H^{b+m-r}	\bigl ( \ftgqOMSc (Y, Y \cap C_{\dB}) \bigr )  \rightarrow \R \\
		 \bigl ( [\omega], [\eta] \bigr )  \mapsto \int_{p^{-1} (Y)} \omega \wedge \eta 
	\end{gathered}
\]
is nondegenerate for \( Y = U, V, U \cap V \), so it is for \( Y = U \cup V \).
\end{lemma}

\begin{proof} The same arguments as in the proof of \cite[Lemma 5.10]{BandR} apply. To establish the short exact sequence for the complexes with compact support, one has to check that for \( \omega \in \ftgqOMSc (U, U \cap \dB) \) and \( f \in C^\infty (U) \) it holds that \[ p^* (f) ~ \omega \in \ftgqOMSc (U , U \cap C_{\dB}). \]
Since this is not very hard, we skip the argument.
\end{proof}

\begin{rem} \label{rem:CSonB} 
	Note, that the compactness of \( B \) implies that
	\[ \ftgOMSc (B, C_{\dB}) = \ftgOMS ( B , C_{\dB}). \] 
	The proof is literally the same as the proof of \cite[Lemma 5.11]{BandR}.
\end{rem}

Together with the bootstrap principle of the above Lemma \ref{lemma:bootsmapMS}, we need an induction basis for being able to use the inductive Mayer-Vietoris argument.
\begin{lemma} (Local Poincar\'e Duality) \label{lemma:localPDMS} \\
	For \( U \in \mathcal{U} \) a coordinate chart, the bilinear form
	\[
		\int : H^r \bigl ( \ftsOMS (U) \bigr ) \times H^{b+m-r} \bigl ( \ftgqOMSc ( U, U \cap C_{\dB} ) \bigr ) \rightarrow \R,
	\]
	where again \( b = \dim B \), \( m = \dim L \), is nondegenerate.
\end{lemma}
\begin{proof}
The map
\[
\begin{split}
	\int : H^r \bigl (\tau_{<K} \Omega^\bullet (L) \bigr ) & \rightarrow H^{m-r} \bigl ( \tau_{\geq K^*} \Omega^\bullet (L) \bigr )^\dagger \\
	[ \omega ] & \mapsto \int_L \_ \wedge \omega
\end{split}
\]
is an isomorphism by \cite[Lemma 5.7]{BandR}. To conclude the argument, we use the results of the Lemmata \ref{lemma:PLfibtr} and \ref{lemma:PLrc} and the commutativity of the following diagram.
\[
\begin{tikzcd}
	H^r \bigl ( \ftsOMS (U) \bigr ) \ar{d}{\int} &~	& H^r \bigl ( \tau_{<K} \Omega^\bullet (L) \bigr ) \ar{ll}{(\phi \circ \pi_2)^*}[swap]{\cong} \ar{d}{\int}[swap]{\cong} \\
	H^{b+m-r} \bigl ( \ftgqOMSc (U, U \cap C_{\dB} ) \bigr )^\dagger & ~ &H^{m-r} \bigl ( \tau_{\geq K^*} \Omega^\bullet (L) \bigr )^\dagger \ar{ll}{({\pi_2}_* \circ {\phi^{-1}}^*)^\dagger}[swap]{\cong} 
\end{tikzcd}
\]
\end{proof}
Now we have all the tools to establish Proposition \ref{prop:PD_ftMS}\\
\emph{Proof of Proposition \ref{prop:PD_ftMS}:}
By Remark \ref{rem:CSonB}, the statement of the proposition is equivalent to the statement that integration induces a map
\[ \int: H^r \bigl ( \ftsOMS (B) \bigr ) \times H^{b+m-r} \bigl ( \ftgqOMSc (B, C_{\dB}) \bigr ) \rightarrow \R  \]
that is nondegenerate for all \( r \).\\
In fact, we prove that the bilinear map
\[ \int: H^r \bigl ( \ftsOMS (U) \bigr ) \times H^{b+m-r} \bigl ( \ftgqOMSc (U, U \cap C_{\dB}) \bigr ) \rightarrow \R \]
is nondegenerate for all \( r \) and all open subsets \( U \subset B \) of the form
\(
	U = \bigcup_{i=1}^s U_{\alpha_0^i ... \alpha_{p_i}^i},
\)
with \( s \leq |I| \) and $U_{\alpha_0^i ... \alpha_{p_i}^i} := U_{\alpha_0^i} \cap \ldots \cap U_{\alpha_{p_i}^i}$, by an induction on \( s \). \\
For \( s = 1 \) the statement was already proven in Lemma \ref{lemma:localPDMS}. 
The induction step follows from the bootsmap principle of Lemma \ref{lemma:bootsmapMS}, compare to \cite[Prop. 5.12]{BandR}.
This finishes the proof, since \( B \) is the finite union \( B = \bigcup_{\alpha \in I} U_\alpha \).
\BE

\subsection{Integration on \texorpdfstring{$ \widetilde{\Omega I}^\bullet_{\pp}(M) $}{partial Omega I}}

\begin{lemma}
For any \( r \in \mathbb{Z} \), integration defines a bilinear form 
\[
	\int : \Omega_{E \MS}^r (M) \times \Omega_{E \MS}^{n-r} (M) \rightarrow \R.
\]
\end{lemma}
\begin{proof} Bilinearity is obvious and the finiteness of the integral is ensured by the compactness of \( M \).
\end{proof}

\begin{cor}
	For any \( r \in \mathbb{Z} \), integration defines bilinear forms 
	\[
		\int :	\widetilde{\Omega I}_{\pp}^r (M) \times \widetilde{\Omega I}_{\pp}^{n-r} (M,C_W) \rightarrow \R .
	\]
\end{cor}

To be able to prove Poincar\'e duality for \( \widetilde{\OI} (M) \) we need two technical lemmas:
\begin{lemma}
	\label{lemma:tl1}
	For \( \nu_0 \in ft_{\geq K} \Omega_{\MS}^{r-1} (B) \) and \( \eta_0 \in ft_{\geq K^*} \Omega_{\MS}^{n-r} (B, C_{\dB}) \) we have 
	\[
		\int_{E} \nu_0 \wedge \eta_0 = 0.
	\]
\end{lemma}
\begin{proof} The proof is literally the same as the proof of \cite[Lemma 7.2]{BandR}.
\end{proof}

\begin{lemma} \label{lemma:tl2}
	For \( \nu \in \widetilde{\Omega I}_{\pp}^{r-1}(M), ~ \eta \in \widetilde{\Omega I}_{\qq}^{n-r} (M,C_W) \) we have
	\[
		\int_M d( \nu \wedge \eta ) = 0.
	\]
\end{lemma}
\begin{proof} 
The boundary of \( M \) is
	\[
		\partial M = E \cup_{\dE} W .
	\]
Recall that $j_W: W \hookrightarrow M$ denotes the embedding of the boundary part $\dW \subset \dM$ in $M$. To prove the lemma, we compute
\[
\begin{split}
	\int_{M} d (\nu \wedge \eta ) & = \int_{\partial M} (\nu \wedge \eta)|_{\partial M} \quad \text{by Stokes' Theorem} \\
	& = \int_E \nu_0 \wedge \eta_0  +  \int_{W} j_W^* (\nu \wedge \eta ) \\
	& \qquad	\text{for some} ~ \nu_0 \in ft_{\geq K} \Omega_{\MS}^{r-1} (B), ~ \eta_0 \in ft_{\geq K^*} \Omega_{\MS}^{n-r} (B, C_{\dB}) \\
	& = 0 + \int_{W} j_W^* (\nu) \wedge j_W^* (\eta) \quad \text{by Lemma} ~ \ref{lemma:tl1}\\
		& = 0 \quad \text{since} ~ \eta \in \widetilde{\Omega I}_{\qq}^{n-r} (M,C_W).
\end{split}
\]
\end{proof}

\subsection{Poincar\'e Duality for \texorpdfstring{$ \widetilde{\Omega I}^\bullet_{\pp}(M) $}{partial Omega I}}
\begin{prop} \label{prop:intonwOI}
	For any \( r \in \mathbb{Z} \), integration on \( \wOI (M) \) induces a bilinear form 
	\[
	\begin{split}
		\int : \widetilde{ HI}_{\pp}^r (M) \times \widetilde{HI}_{\qq}^{n-r} (M,C_W) & \rightarrow \R \\
			\bigl ( [\omega] , [ \eta ] \bigr ) & \mapsto \int_M \omega \wedge \eta.
	\end{split}
\]
\end{prop}
\begin{proof} Let \( \omega \in \widetilde{\Omega I}_{\pp}^r (M) \) closed, \( \widetilde{\omega} \in \widetilde{\Omega I}_{\pp}^{r-1} (M) \), \( \eta \in \widetilde{\Omega I}_{\qq}^{n-r} (M,C_W) \) closed and \( \widetilde{\eta} \in \widetilde{\Omega I}_{\qq}^{n-r-1} (M,C_W) \).
\[
	\int_M (\omega + d \widetilde{\omega} ) \wedge \eta = \int_M \omega \wedge \eta + \int_M d( \widetilde{\omega} \wedge \eta ) = \int_M \omega \wedge \eta,
\]
where the last step holds by the previous Lemma \ref{lemma:tl2}.
By an analogous argument
\[
	\int_M \omega \wedge (\eta + d \widetilde{\eta} ) = \int_M \omega \wedge \eta.
\]

\end{proof}

\begin{lemma} \label{lemma:EMSqis}
	The subcomplex inclusion 
	\[
		\OEMS (M) \hookrightarrow \Omega^\bullet (M)
	\]
	is a quasi-isomorphism.
\end{lemma}
\begin{proof} As in Lemma \ref{lemma:MSdc}, we can apply the arguments of \cite[Prop. 2.4]{BandR}, integrating forms in the collar direction.
\end{proof}

\vspace{0.5cm} 
\begin{prop}
	Let \( N := M - \partial M \) and let $C = C_E \cup C_W$ be the union of the two collar neighbourhoods and define the subcomplex $\Omega^\bullet_{rel} (N)$ of forms on $N$ that vanish on $C$ as follows.
	\[ \Omega_{rel}^\bullet (N) = \Omega^\bullet (N, C) = \{ \omega \in \Omega^\bullet (N) \,|\, \omega|_{C\cap N} = 0\}. \] Then the subcomplex inclusion
	\[
		\Omega_{rel}^\bullet (N) \hookrightarrow \Omega_c^\bullet (N)
	\]
	is a quasi-isomorphism.
\end{prop}
\begin{proof} We factor the subcomplex inclusion $\Omega_{rel}^\bullet (N) \hookrightarrow \Omega_c^\bullet (N)$ as
\[
\Omega_{rel}^\bullet (N) \hookrightarrow \Omega_c^\bullet (N,C_E) \hookrightarrow \Omega_c^\bullet (N), 
\]
with $\Omega_c^\bullet (N, C_E) = \Omega_c^\bullet (N) \cap \Omega^\bullet (N, C_E).$ We use a standard argument to prove that both subcomplex inclusions are quasi-isomorphisms. Let $\epsilon > 0$ be a small number and let
\( 	\widetilde{C}_X  := c_X \bigl ( (0,1+\epsilon) \times X \bigr ) \subset N \)
denote slightly larger collar neighbourhoods of $X= E,W$ in $N.$ 
Let \( \xi_X\) be a smooth cutoff function in collar direction on \( N \) with \( \xi|_{N- \widetilde{C}_X} = 0 \) and \( c_X^* \xi =1 \), where, again, $X=E,W$.

We first prove that $H_{rel}^\bullet (N) \rightarrow H_c^\bullet (N,C_E)$ is an isomorphism.\\
\underline{Injectivity:} Let \( \omega \in \Omega_{rel}^\bullet (N ) \) be a closed form such that \( \omega = d \widehat{\eta} \) for some \( \widehat{\eta} \in \Omega_c^\bullet (N,C_E) \). To prove injectivity, we have to show that the cohomology class \( [\omega] \in H_{rel}^\bullet (N) \) is also zero. We first decompose \( \widehat{\eta}|_{\widetilde{C}_W} \) into its tangential and normal component:
$ \widehat{\eta}|_{\widetilde{C}_W} = \widehat{\eta}_T (t) + dt \wedge \widehat{\eta}_N (t). $
We then define a new form
\[
	\eta := \widehat{\eta} - d \bigl ( \xi_W \int_{0}^t \widehat{\eta}_N (\tau) d\tau \bigr ) \in \Omega^\bullet (N).
\]
This new form satisfies $d\eta = d\widehat{\eta} = \omega$ and has a vanishing normal component, since
\[
c_W^* \eta = \underbrace{c_W^* \widehat{\eta} - dt \wedge \widehat{\eta}_N (t)}_{= \widehat{\eta}_T (t)} - \xi_W \int_0^t d_W \widehat{\eta}_N (\tau) d\tau =: \eta_T (t).
\]
We want to show, that $c_W^* \eta = 0.$ By assumption on $\omega,$ we get the following.
\[ 
	0 = c_W^* \omega = c_W^* \left( d \eta \right) = d (c_W^* \eta) = d_W \eta_T (t)  + dt \wedge \eta_T' (t) .
 \]
Here, $d_W$ denotes the part of the boundary operator on $C_W$ with the derivatives along $W$ and $\eta_T ' (t)$ is the derivative of $\eta_T$ in the collar direction.
Hence, \( \eta_T' (t) = 0 \), i.e. \( \eta_T \) is independent of the collar coordinate (for $t<1$). Since $\widehat{\eta}$ has compact support in $N,$ there is a $\delta > 0,$ such that $\widehat{\eta}_T (t) \equiv 0$ and $\widehat{\eta}_N (t) \equiv 0$ for all $t < \delta.$ Therefore, also $\eta (t) = \eta_T (t) \equiv 0$ for all $t<\delta$. he fact that $\eta_T$ does not depend on $t$ then gives that $c_W^* \eta  \equiv 0$ and therefore $\eta \in \Omega_{rel}^\bullet (N).$ %
\\
\underline{Surjectivity:} Let \( \widehat{\omega} \in \Omega_c^\bullet (N,C_E) \) be a closed form. We want to show that there is a closed form \( \omega \in \Omega_{rel}^\bullet (N ) \) and a form \( \widehat{\eta} \in \Omega_c^\bullet (N,C_E) \) such that \( \widehat{\omega} = \omega + d \widehat{\eta} \). As in the previous step, we decompose
\( \widehat{\omega}|_{\widetilde{C}_W} = \widehat{\omega}_T (t) + dt \wedge \widehat{\omega}_N (t) \)
and define the closed form $\omega \in \Omega^\bullet (N)$ by
\[ \omega := \widehat{\omega} - d \bigl ( \xi_W \int_{0}^t \widehat{\omega}_N (\tau) d\tau \bigr ).
\]
As before, the normal part of \( c_W^* \omega \) is zero and since \( \omega \) is closed, \( \omega_T ' (t) = 0 \) for $t<1.$ 
Also in analogue to the injectivity part, the fact that $\widehat{\omega}$ has compact support implies that there is a $\delta > 0$ with $c_W^* \omega (t) = 0 $ for all $t < \delta$ and we can deduce that \( \omega \in \Omega_{rel}^\bullet (N ).$ 
On the other hand, the relation $\widehat{\omega}_N (t) = 0 $ for every $t < \delta$ implies that the support of the form $\widehat{\eta} := \xi_W \int_0^t \widehat{\omega}_N (\tau) \, d\tau $ is compact.

\vspace{1ex}
The proof of the quasi-isomorphy of $\Omega_c^\bullet (N, C_E) \hookrightarrow \Omega_c (N)$ uses the same argument. The only difference is that we work on the collar neighbourhood of $E$ instead of $W$, use $\xi_E$ instead of $\xi_W$ and so on.
\end{proof}
\noindent
Finally we are able to prove Poincar\'e Duality for \( \wOI (M) \):
\begin{thm}(Poincar\'e duality for \( \widetilde{HI}_{\pp} (M) \)) \label{prop:PDonwOI} \\ 	
	For any \( r \in \mathbb{Z} \), the bilinear form 
\[
	\int : \widetilde{HI}_{\pp}^r (M) \times \widetilde{HI}_{\qq}^{n-r} (M,C_W) \rightarrow \R
\]
of Proposition \ref{prop:intonwOI} is nondegenerate.
\end{thm}
\begin{proof} \emph{First Step} 
Recall the definitions of the complexes $\OEMS$ in Definition \ref{def:EMS} and $\Omega_{rel}^\bullet (M)$ in Definition \ref{def:relderham}.
By the previous Lemma \ref{lemma:EMSqis}, the subcomplex inclusion \( \OEMS (M) \subset \Omega^\bullet (M) \) induces an isomorphism
$	H^r_{E \MS} (M) := H^r \bigl (\OEMS (M) \bigr ) \xrightarrow{\cong} H^r (M)$ 
for any \( r \in \mathbb{Z} \). The inclusion \( i : N \hookrightarrow M \) is a homotopy equivalence and hence induces an isomorphism
$i^*: H^r (M) \xrightarrow{\cong} H^r (N)$
for all \( r \in \mathbb{Z} \), as well as the isomorphism
$	i^*: H^r_{rel} (M) \xrightarrow{\cong} H_{rel}^r (N).$
Since $N$ is an open manifold, integration gives an isomorphism
\[
	\int : H^r(N) \stackrel{\cong}{\longrightarrow} H^{n-r}_c (N)^\dagger
\]
for all \( r \in \mathbb{Z} \). Combining these maps, we get the following commutative diagram for any $r\in\Z.$
\[ \begin{tikzcd}
	H^r_{E \MS} (M) \ar{r}{\cong} \ar{rd}{\int} & H^r (M) \ar{d}{\int} \ar{rr}{i^*}[swap]{\cong} & \ & H^r(N) \ar{d}{\int}[swap]{\cong}\\
	&	H^{n-r}_{rel} (M)^\dagger &  H^{n-r}_{rel} (N)^\dagger \ar{l}{{i^*}^\dagger}[swap]{\cong} & H_c^{n-r} (N)^\dagger \ar{l}{\text{incl}_*^\dagger}[swap]{\cong}
\end{tikzcd}
\]
Hence, integration induces an isomorphism
\[
	\int : H_{E \MS}^r (M) \xrightarrow{\cong} H_{rel}^{n-r} (M )^\dagger 
\]
for all \( r \in \mathbb{Z} \).

\emph{Second Step} 
By Proposition \ref{prop:PD_ftMS}, integration gives an isomorphism
\[
	\int : H^r \bigl ( \ftsOMS (B) \bigr ) \stackrel{\cong}{\longrightarrow} H^{n-r-1} \bigl ( \ftgqOMS (B, C_{\dB}) \bigr )^\dagger .
\]

\emph{Third Step}
The distinguished triangles of the two Lemmata \ref{lemma:disttri_tilde_1} and \ref{lemma:disttri_tilde_2} give the long exact sequences on cohomology, which fit into the following diagram.
\[ \begin{tikzcd}
		\vdots \ar{d} 									& 	\vdots \ar{d} 	\\
		H^{r-1} \bigl ( \ftsOMS (B) \bigr ) \ar{r}{\int}[swap]{\cong} \ar{d}{\delta} 	& H^{n-r} \bigl ( \ftgqOMS (B, C_{\dB}) \bigr )^\dagger \ar{d} \\
		\widetilde{HI}^r_{\pp} (M) \ar{r}{\int} \ar{d} 					& \widetilde{HI}^{n-r}_{\qq} (M,C_W)^\dagger \ar{d} \\
		H^r_{E \MS} (M) \ar{r}{\int}[swap]{\cong} \ar{d}{Q} 				& H_{rel}^{n-r} (M)^\dagger \ar{d}{D^\dagger} \\
		H^r \bigl ( \ftsOMS (B) \bigr ) \ar{r}{\int}[swap]{\cong} \ar{d} 		& H^{n-r-1} \bigl ( \ftgqOMS (B, C_{\dB}) \bigr )^\dagger \ar{d} \\
		\vdots 										& \vdots
	\end{tikzcd}
\]
Together with the 5-Lemma, proving that this diagram commutes (up to sign) establishes the desired result.
We first prove that the top square (TS) in the diagram commutes and therefore describe the connecting homomorphism 
$ \delta: H^{r-1} \bigl ( \ftsOMS (B) \bigr ) \rightarrow \widetilde{HI}_{\pp}^r (M). $
Let \( \omega \in ft_{<K} \Omega_{\MS}^r (B) \) closed, i.e. \( d \omega = 0 \). Then \( d \gamma_B \omega = 0 \) holds as well, where \( \gamma_B : \ftsOMS (B) \rightarrow \OMS (B)/\ftgOMS (B) \) is the quasi-isomorphism defined in equation (\ref{eq:gammaB}).
Since the map
$	\overline{i} \circ \overline{J}_E^* : Q_E^\bullet (M) \rightarrow Q^\bullet (B) $
defined in Lemma \ref{lemma:disttri_tilde_1} is a quasi-isomorphism, there is a closed form \( \overline{\omega} \in Q_E^r (M) \) as well as a form $\chi \in (\OMS (B) / \ftgOMS (B) )^r$ such that  $ \overline{i} \circ \overline{J}_E^* \overline{\omega} = \gamma_B \omega + d \chi\). 
Let \( \xi \in \Omega_{E \MS}^r (M) \) be a representative of \( \overline{\omega} \) , i.e. \( q (\xi) = \overline{\omega} \). Then \( d \xi \in \wOI (M) \) since
$	q ( d \xi ) = d q( \xi ) = d \overline{\omega} = 0.$
Hence \( ( - d \xi , \xi ) \in C^r (i) \), the mapping cone of the subcomplex inclusion \( i : \wOI (M) \hookrightarrow \Omega_{E \MS}^\bullet (M) \) with 
\( d ( -d \xi, \xi ) = 0 \). Therefore by the definition of distinguished triangles and the induced long exact cohomology sequences,
$
	\delta [\omega] = [- d \xi].$
The relation
$q( \text{incl} ~ \omega ) = \gamma_B (\omega) = \overline{i} \circ \overline{J}_E^* q (\xi) = q ( J_E^* \xi ) = q (\sigma_E^* c_E^* \xi) \in Q^r (B) $
implies that
$	\alpha := \sigma_E^* c_E^* \xi - \omega \in \ftgOMS (B).$ \\
Now, let \( \omega \in ft_{<K} \Omega_{\MS}^{r-1} (B), ~ \eta \in \widetilde{\Omega I}_{\qq}^{n-r} (M,C_W) \) be closed forms and $C_E^{1/2} := c_{E} ( E \times [0,\frac{1}{2}) )$. By the prior arguments, we get the following.
\[ \begin{split}
		\int_M \delta (\omega) \wedge \eta & = - \int_M d \xi \wedge \eta = - \int_M d (\xi \wedge \eta) \\
			& = - \int_{M - C_E^{1/2}}  d (\xi \wedge \eta) - \int_{C_E^{1/2}} d (\xi \wedge \eta) \\
					& = - \int_{M-C_E^{1/2}} d (\xi \wedge \eta),
\end{split}
\]
The second integral in the second line vanishes, since 
\[ d (\xi \wedge \eta)|_{C_E^{1/2}} = (c_E^{-1})^* \pi_E^* d (\xi_0 \wedge \eta_0) \] for some \( \xi_0 \in \Omega_{MS}^r (B) ,\) \( \eta_0 \in ft_{\geq K^*} \Omega_{\MS}^{n-r} (B, C_{\dB}) \). But
\[
	\int_{C_E} d (\xi \wedge \eta) = \text{const.} \cdot \int_E d (\xi_0 \wedge \eta_0 ) = 0
\]
as an integral of an \( n \)-form over a \( (n-1)\)-dimensional manifold. Let $C_{\dW}^{1/2} := c_{\dW} ([0, \frac{1}{2}))$ and \( J_W: W - C_{\dW}^{1/2} \hookrightarrow W \hookrightarrow M \). Then by Stokes' Theorem for manifolds with corners (see \cite[Theorem 16.25]{Lee}), we get:
\[ \begin{split}
		-\int_{M - C_E^{1/2}} d (\xi \wedge \eta) & = - \int_E \sigma_E^* c_E^* \xi \wedge \widetilde{J_E}^* \eta + \int_{W-C_{\dW}^{1/2}} J_W^* (\xi \wedge \eta) \\
			& = - \int_{E} \omega \wedge \widetilde{J_E}^* \eta - \int_E \alpha \wedge \widetilde{J_E}^* \eta = - \int_E \omega \wedge \widetilde{J_E^*} \eta, 
\end{split}
\]
where 
\[
	\int_{W-C_{\dW}^{1/2}} J_W^* (\xi \wedge \eta) = 0, 
\]
since \( \eta|_{C_W} = 0 \), and 
\[
\int_E \alpha \wedge \widetilde{J_E}^* \eta  = 0
\]
by Lemma \ref{lemma:tl1}. Note, that we used a slight abuse of notation in the top line of the equation: The boundary of $M - C_E^{1/2}$ is 
$ c_{E} (E \times \{ \frac{1}{2} \} ) \cup (W - C_{dW}^{1/2}).$
But since $c_E^* d (\xi \wedge \eta)$ is, by definition, the pullback of a form on $E \times \left\{ 0 \right\},$ the pullback of this form to $c_{E} (E \times \left\{ \frac{1}{2} \right\} ) $ is the same form as the pullback to $E$. Hence, we write $\int_E$ in this line.
The upshot of our calculations is, that (TS) commutes up to sign.

\vspace{0.2cm}
\noindent
The commutativity of the middle square (MS) is obviously fullfilled since both the vertical maps are induced by the subcomplex inclusions \( \wOI (M) \hookrightarrow \OEMS (M) \) and 
\( \Omega_{rel}^\bullet (M) \hookrightarrow \wOIq (M,C_W) \).

\vspace{0.2cm}
\noindent
To prove the commutativity of the bottom square (BS), we first investigate the connecting homomorphism
$ D: H^{n-r-1} \bigl ( \ftgqOMS (B, C_{\dB}) \bigr ) \rightarrow H_{rel}^{n-r} (M).$
We look at the distinguished triangle (\ref{disttri_2}).\\
For \( \eta \in ft_{\geq K^*} \Omega_{\MS}^{n-r-1} (B, C_{\dB}) \) closed, the surjectivity of \( \widetilde{J_E}^* \) implies that there is a form \( \overline{\eta} \in \widetilde{\Omega I}_{\qq}^{n-r-1} (M, C_W) \) such that \( \widetilde{J_E}^* \overline{\eta} = \eta \). Since \( \widetilde{J_E}^* \) is a chain map, \( d \overline{\eta} \in \ker \widetilde{J_E}^* = \Omega_{rel}^{n-r} (M) \). Let \( \rho : \Omega_{rel}^\bullet (M) \hookrightarrow \wOIq (M, C_W) \) denote the subcomplex inclusion and \( C^\bullet (\rho) \) its algebraic mapping cone. Then the map 
$ f: C^\bullet (\rho) \rightarrow \ftgqOMS (B, C_{\dB}), ~ (\tau, \sigma) \mapsto \widetilde{J_E}^* (\sigma) $
is a quasi-isomorphism (by the standard argumentation). The cocycle
$ c := ( - d \overline{\eta}, \overline{\eta} ) \in C^{n-r-1} (\rho) $
satisfies the equation \( f(c) = \widetilde{J_E}^* \overline{\eta} = \eta \) and hence \( D [\eta] \) can be described as
$	D [\eta] = [ - d \overline{\eta} ].$
We next describe the map 
\[
	Q: H^r \bigl ( \OEMS (M) \bigr ) \rightarrow H^r \bigl( \ftsOMS (B) \bigr),
\]
induced by the corresponding map in the distinguished triangle (\ref{disttri_1}). \\
Let \( \omega \in \Omega_{E \MS}^r (M) \) be a closed form. Then \( J_E^* \omega \in \Omega_{\MS}^r (B) \) represents the image of \( \omega \) under the composition
\[
	\overline{J}_E^* \circ q: \; \OEMS ( M) \rightarrow Q_E^\bullet (M) \rightarrow Q^\bullet (B). 
\]
Since \( \gamma_B: \ftsOMS (B) \rightarrow Q^\bullet (B) \) is a quasi-isomorphism, there are forms \( \overline{\omega} \in ft_{<K} \Omega_{\MS}^r (B) \), \( d \overline{\omega} = 0 \), and \( \xi \in \Omega_{\MS}^{r-1} (B) \) such that
$	\gamma_B (\overline{\omega}) = \overline{J}_E^* q (\omega) + d q_B (\xi) .$
The above map \( Q \) is then given by $ Q [\omega] = [\overline{\omega}]. $
Note that the form \( \alpha := \overline{\omega} - J_E^* \omega - d \xi \) is contained in \( \ftgOMS (B) \). We can now verify the commutativity of (BS) by proving
\begin{equation} \label{eq:tildeBS}
	\int_{M} \omega \wedge (- d \overline{\eta} ) = \pm \int_{E} \overline{\omega} \wedge \eta,
\end{equation}
with \( [- d \overline{\eta} ] = D[\eta] \) and \( [\overline{\omega} ] = Q [\omega ] \).
\[ \begin{split}
		\int_M \omega \wedge (- d \overline{\eta}) &= - \int_M \omega \wedge d \overline{\eta} = \pm \int_{M - C_E} d( \omega \wedge \overline{\eta}) - \int_{C_E} \omega \wedge d \overline{\eta} \\
		& = \pm \int_{M - C_{E}} d ( \omega \wedge \overline{\eta}) \quad (\text{since} ~ d \overline{\eta} \in \Omega_{rel}^\bullet (M)) \\
			& = \pm \int_E J_E^* \omega \wedge \widetilde{J_E}^* \overline{\eta} \pm \int_{W- C_{\dW}} J_W^* ( \omega \wedge \overline{\eta} ) \quad \text{(Stokes)} \\
			& = \pm \int_E (\overline{\omega} - \alpha - d \xi) \wedge \widetilde{J_E}^* \overline{\eta} \quad (\text{above} ~+ 
			~ \eta|_{C_W} = 0) \\
			& = \pm \int_E \overline{\omega} \wedge \eta,
\end{split}
\]
since \( \widetilde{J_E}^* \overline{\eta} = \eta \), 
\(
	\int_E \alpha \wedge \eta = 0
\)
by Lemma \ref{lemma:tl1} and
\[
	\int_E d \xi \wedge \eta = \int_E d ( \xi \wedge \eta ) 
	= \int_{\dE} \xi \wedge \eta = 0
\]
by Stokes' Theorem and since \( c_W^* \overline{\eta} = 0 \) and hence \( j_{\dE}^* \eta = 0 \).
Thus, also (BS) commutes up to sign and the theorem is proven.
\end{proof}

\section{The de Rham Intersection Complex \texorpdfstring{$ \Omega I^\bullet_{\bar{\lowercase{p}}}(M) $}{}} \label{OI_3strata}


\subsection{Truncation and Cotruncation of \texorpdfstring{$ \OI (W) $}{The Intersection Form Complex on W}} \label{subs:CotrOI}
We first note some obeservations about the boundary part \( W \subset \dM \) and the pullback of the forms in $\wOI (M)$ to $W$.
\begin{rem}\label{rem:reformOI}
	The \( (n-1) \)-dimensional compact manifold with boundary \( W \) is the blowup of the singular stratified space \( \partial X' = L_0, \) the link of $X_0$ mentioned in Section \ref{subs:3strata}. The boundary \( \dW \) of \( W \) is the total space of the flat link bundle \( q: \dW = \dE \to \dB \), with \( B = \Sigma \) the bottom stratum of the stratified pseudomanifold-with-boundary \( X' \). 
	Hence, following \cite{BandR},  we can construct the chain complex of intersection forms \( \OI (W) \) as a subcomplex of the complex of differential forms on $W$:
\[ \OI (W) := \left\{ \omega \in \Omega^\bullet (W) | c_{\dW}^* \omega = \pi_{\dW}^* \eta, ~ \text{for some}~ \eta \in \ftgOMS (\dB) \right\}. \]
We claim that the pullback of a form $\omega \in \wOI (M)$ to the boundary part $W$ is constained in $\OI (W)$. So let $\omega \in \wOI (M).$ Then $c_E^* \omega = \pi_E^* \eta $ for some $\eta \in \ftgOMS (B)$ and $j_{\dE}^* \eta \in \ftgOMS (\dB)$. This allows us to rewrite $c_{\dW}^* j_W^* \omega \in \Omega^\bullet (\dW \times [0,1)) $ as follows.
\begin{align*}
c_{\dW}^* j_W^* \omega &= j_W^* c_E^* \omega = j_W^* \pi_E^* \eta = \pi_{\dE}^* j_{\dE}^* \eta = \pi_{dE}^* j_{\dE}^* \eta. 
\end{align*}
By definition, we then have $j_W^* \omega \in \OI (W).$
\end{rem}
Since the pullback of forms in $\wOI (M)$ to $W$ is contained in $\OI (W),$ we cotruncate $\OI (W)$ in degree $L$ to define the intersection form complex $\OI (M)$ and prove Poincar\'e duality for the intersection space cohomology groups $HI_{\pp}^\bullet (X) = H^\bullet \left( \OI (M) \right).$ See Section \ref{subs:CotrValues} for the definition of the cotruncation value $L$.

\vspace{1ex}
By a cotruncation of a cochain complex $(C^\bullet, d)$ of vector spaces over some field $\K$ in degree $L \in \Z$ we mean a subcomplex $\iota: \tau_{\geq L} C^\bullet \hookrightarrow C^\bullet$ such that the subcomplex inclusion induces an isomorphism on cohomology $\iota^*: H^r \left( \tau_{\geq L} C^\bullet, d \right) \xrightarrow{\cong} H^r \left( C^\bullet, d \right)$ in degrees $r \geq L$ and $H^r (\tg C^\bullet) = 0$ for $r < L.$ 
Note that for our purposes the standard cotruncation 
\[ \ldots \to 0 \to \text{coker} \, d \to C^{L+1} \to C^{L+2} \to \ldots \]
is not suitable since it does not give rise to a subcomplex of $C^\bullet$ (although it satisfies the cohomology conditions).

\begin{ex}[Cotruncation of differential forms on a closed manifold]\label{ex:CotrHodgeClosed}
Recall that on a closed Riemannian manifold \( M \) the Hodge decomposition provides orthogonal splittings
\begin{equation}
	\Omega^r (M) = \im d \oplus \ker d^* = \im d \oplus \im d^* \oplus  \mathcal{H}^r (M) 
	\label{eq:Hodgedecomp_closedmfd}
\end{equation}
with \( \mathcal{H}^r (M) := \bigl \{ \omega ~ \bigl | ~ \Delta \omega = (d ~ d^* + d^* ~ d) \omega = 0 \bigr \} \) the harmonic \(r\)-forms on \( M\). This allows us to define the cotruncated subcomplex of smooth differential forms:
\[
	\tg \Omega^\bullet (M) : = ... \rightarrow 0 \rightarrow \ker d^* \rightarrow \Omega^{L+1} (M) \rightarrow \Omega^{L+2} (M) \rightarrow ... \subset \Omega^\bullet (M).
\]
By the Hodge decomposition this complex is a cotruncation of the complex $\Omega^\bullet (M)$ of smooth forms on $M$. In addition, the following direct sum decomposition holds.
\[
	\tg \Omega^\bullet (M) \oplus \ts \Omega^\bullet = \Omega^\bullet (M),
\]
where \( \ts \Omega^\bullet (M) := ... \rightarrow \Omega^{L-1} (M) \rightarrow \im d \rightarrow 0 \rightarrow ... \subset \Omega^\bullet (M) \), which is a truncation of $\Omega^\bullet (M).$ More details for this construction can be found in \cite[p.18]{BandR}.
\end{ex}

\begin{rem}[Cotruncation using Hodge theory]
As outlined in Section \ref{subs:DepthOne}, Banagl defines the intersection form complex $\OI (M)$ for depth one pseudomanifolds as the complex of forms on $M$ with restriction to some open neighbourhood of the boundary equaling the pullback of some fiberwisely cotruncated differential form on $\dM$. To define cotruncation on the link he uses the real Hodge decomposition as explained in Example \ref{ex:CotrHodgeClosed}. The advantage of this approach is, that he can give a condition on the fiber bundle such that the cotruncation can be imposed fiberwisely on multiplicative forms: The bundle has to be geometrically flat, i.e. the transition functions have to be isometries. 

With that in mind, it would be natural to use the generalization of the Hodge decomposition theorem to compact manifolds with boundary to cotruncate the complex of differential forms there. Combining the Hodge-Morrey and the Friedrichs decomposition on a compact manifold $W$ with boundary $\dW,$ one gets the orthogonal direct sum decomposition
\[ \Omega^r (W) = \im d^{r-1} \oplus cC_N^r (W), \]
for any $r \in \Z,$ where $cC_N^r (W) := \bigl \{ \omega \in \Omega_N^r (W) \bigl | d^* \omega = 0 \bigr \}$ is the vector space of coclosed $r$-forms that satisfy the Neumann boundary condition, i.e. with vanishing normal component near the boundary. This decomposition follows from \cite[Corollary 2.4.9]{Schwarz} and Green's formula, \cite[Prop. 2.1.2]{Schwarz}. It can be used to define a cotruncation on the complex $\Omega^\bullet (W)$ by
\[ \widetilde{\tau}_{\geq L} \Omega^\bullet (W) := \ldots \to 0 \to cC_N^L (W) \to \Omega^{L+1} (W) \to \Omega^{L+2} (W) \to \ldots. \]
Again, if $W$ were the link of a flat fiber bundle, the transition functions that are isometries on $W$ would map cotruncated forms to cotruncated forms. The reason for not using this construction is that it is dubious, whether intersecting this complex with $\OI (W)$ gives a cotruncation of $\OI (W)$ or not. That means that it is unclear whether the Hodge-Morrey-Friedrichs decomposition also produces a complement of $d^{L-1} \left( \Omega I_{\pp}^{L-1} (W) \right)$ in $\Omega^L (W)$ (note, that this would imply that the subcomplex inclusion $\iota: \OI (W) \hookrightarrow \Omega^\bullet (W)$ induced an injection on the $L$-th cohomology groups). Therefore we have to use an alternative approach to cotruncate $\OI (W)$ in degree $L$. 
\end{rem}

\begin{lemma}[Cotruncation of $\OI (W)$ in degree $L$]\label{lemma:CotrOIW}
Choose a complement $\mathfrak{C}^L \subset \OI (W)$ of $d^{L-1} \left( \Omega I_{\pp}^{L-1} (W) \right)$ in $\OI (W)$ and set 
\[ \tg \OI (W) := \ldots \to 0 \to \mathfrak{C}^L \to \Omega I_{\pp}^{L+1} (W) \to \Omega I_{\pp}^{L+2} (W) \to \ldots \]
Then $\tg \OI (W) \subset \OI (W)$ is a cotruncation in degree $L$ and, moreover, there is a direct sum decomposition 
\[ \OI (W) = \ts \OI (W) \oplus \tg \OI (W), \]
where 
\[ \ts \OI (W) := \ldots \Omega I_{\pp}^{L-2} (W) \to \Omega I_{\pp}^{L-1} (W) \to d^{L-1} ( \Omega I_{\pp}^{L-1} ) \to 0 \to \ldots \]
 is the standard truncation of $\OI (W)$ in degree $L$.
\end{lemma}
\begin{proof}
It is clear, that $H^r (\tg \OI (W)) = 0 $ for $r<L$ and $H^r (\tg \OI (W) ) \cong H^r (\OI (W))$ for $r \geq L+2.$ The only interesting degrees are $r=L$ and $r = L+1.$ \\
\underline{Degree $r=L$:} $\iota^*: H^L (\tg \OI (W)) \to H^L (\OI (W))$ is injective, since no nonzero cycles in $\mathfrak{C}^L = \tau_{\geq L} \Omega I_{\pp}^L (W)$ are boundaries. Let $\omega \in \Omega I_{\pp}^L (W)$ be closed. Then, since $\Omega I_{\pp}^L (W) = \mathfrak{C}^L + \im d^{L-1},$ $\omega = \eta + d \alpha$ with $\eta \in \mathfrak{C}^L$ closed. Hence $\iota^* [\eta] = [\omega]$ and thus $\iota^*$ is also surjective.\\
\underline{Degree $r=L+1$:} It is obvious, that $\iota^*: H^{L+1} (\tg \OI (W)) \to H^{L+1} (\OI (W) )$ is surjective, since $\tau_{\geq L} \Omega I_{\pp}^{L+1} (W) = \Omega I_{\pp}^{L+1} (W)$. We prove that it is also injective. Let $\omega = d \alpha \in \Omega I_{\pp}^{L+1} (W)$ be a boundary. By the definition of $\mathfrak{C}^L$, there exist $\eta \in \mathfrak{C}^L$ and $\beta \in \Omega I_{\pp}^{L-1} (W)$ with $\alpha = \eta + d \beta.$ Then $\omega = d \alpha = d ( \eta + d \beta) = d \eta$, so $[\omega] = 0 \in H^{L+1} (\tg \OI (W) )$ and $\iota^*$ is injective.\\
The direct sum decomposition of $\OI (W)$ holds, since $\mathfrak{C}^L$ is a complement of $d^{L-1} ( \Omega I_{\pp}^{L-1} (W) )$ in $\Omega I_{\pp}^L (W)$ and in all other degrees either the truncated complex is zero and the cotruncated complex contains all forms or vice versa.
\end{proof}

\subsection{Definition of the de Rham Intersection Complex}
\begin{mydef}[The de Rham intersection complex \( \OI (M) \)]
Take the cotruncated complex $\tg \OI (W)$ of Lemma \ref{lemma:CotrOIW}. We define the intersection form complex $\OI (M)$ on $M$ as follows.
	\[
		\OI (M) := \bigl \{ \omega \in \wOI (M) \bigl |~
			\exists ~ \eta \in \tg \OI (W):
			c_W^* \omega = \pi_W^* \eta
			  \bigr \}.
	\]
\end{mydef}

We start to give the preparational material for the proof of Poincar\'e Duality for \( \OI (M) \):
\begin{lemma} \label{lemma:disttriOI1}
The following sequences are short exact.
\begin{align*}
0 \longrightarrow  \widetilde{\OI} (M,C_W) \longrightarrow  \OI (M) \xrightarrow{j_W^*}  \tau_{\geq L} \Omega I_{\pp}^\bullet (W) \longrightarrow  0 \\
0 \longrightarrow \OI (M) \longrightarrow 		  \wOI (M) \xrightarrow{\proj \circ j_W^*} \ts \OI (W) \longrightarrow 0
\end{align*}
where $\proj: \OI (W) = \ts \OI (W) \oplus \tg \OI (W) \to \ts \OI (W)$ is the projection on the first factor. 
\end{lemma}
\begin{proof} For the first sequence, note that the kernel of the surjective map \( j_W^* : \OI (M) \to \tau_{\geq L} \Omega I_{\pp}^\bullet (W) \) is 
\[ \bigl \{ \omega \in \OI (M) \bigl | c_W^* \omega = 0 \bigr \} = \widetilde{\OI} (M,C_W ). \]
To prove that the second sequence is also exact, note that, by definition, the pullback to $W$ maps forms in $\OI (M)$ to forms in $\tg \OI (W).$ Hence,
\[ \ker \left( \wOI (M) \xrightarrow{j_W^*} \OI (W) \xrightarrow{\proj} \ts \OI (W) \right) = \OI (M). \]
\end{proof}

\begin{lemma}\label{lemma:PDHIW}
Integration induces a nondegenerate bilinear form
	\[
		\int : HI_{\pp}^r (W) \times HI_{\qq}^{n-1-r} (W) \to \R.
	\]
\end{lemma}
\begin{proof} Notice that \( \OI (W) \cong \OI (W - \dW) \) and consider \cite[Theorem 8.2]{BandR}.
\end{proof}

\begin{lemma}
	\label{lemma:PDtrHIW}
	Also integration induces a nondegenerate bilinear form
	\[
		\int : H^r \bigl (\ts \OI (W) \bigr ) \times H^{n-1-r} \bigl ( \tgq \OIq (W) \bigr ) \to \R.
	\]
\end{lemma}
\begin{proof} For \( r \geq L \) we have that \( n-1-r < L^* \) and both complexes are zero and therefore also the cohomology groups. For \( r < L \) we have that \( n-r-1 \geq L^* \) and hence
$	H^r \bigl ( \ts \OI (W) \bigr ) = HI_{\pp}^r (W) $
as well as
$	H^{n-1-r} \bigl ( \tgq \OIq (W) \bigr ) \cong HI_{\qq}^{n-1-r} (W).$
Therefore we traced back the statement of the lemma to the result of the previous Lemma \ref{lemma:PDHIW}.
\end{proof}

\subsection{Integration on \texorpdfstring{$\OI (M)$}{the Intersection de Rham Complex on M}}
\label{subsection:IntonOI}
In analogy to \cite[Lemma 7.1,Cor. 7.2]{BandR} we have
\begin{lemma}
	Integration defines bilinear forms
	\[
		\int : \Omega I_{\pp}^r (M) \times \Omega I_{\qq}^{n-r} (M) \to \R.
	\]
\end{lemma}

The following lemma is the extension of \cite[Lemma 7.4]{BandR} to the 3-strata case:
\begin{lemma} \label{lemma:intexact}
	Let \( \omega \in \Omega I_{\pp}^{r-1} (M), ~ \eta \in \Omega I_{\qq}^{n-r} (M), \) Then 
	\[
		\int_{M} d ( \omega \wedge \eta ) = 0.
	\]
\end{lemma}
\begin{proof} By Stokes' Theorem on manifolds with corners we get:
\[
	\int_M d (\omega \wedge \eta ) = \int_W  j_W^* (\omega \wedge \eta) + \underbrace{\int_E j_E^* (\omega \wedge \eta)}_{= 0 ~ \text{by Lemma \ref{lemma:tl1}}}.
\]
By definition of \( \OI (M) \) and $\OIq (M)$, we have $j_W^* \omega \in \tg \OI (W)$ and $ j_W^* \eta \in \tgq \OIq (W).$
For \( r-1 \geq L = n-1- \pp (n-1) = 2 + \qq (n-1) \) we get \( n-r \leq n-3-\qq (n-1) < n-1 - \qq (n-1) = L^* \) and hence
$j_W^* \eta \in \tgq \Omega I_{\qq}^{n-r} (W) = \left\{ 0 \right\}.$
For \( r-1 < L \) we have $\omega \in \tg \Omega I_{\pp}^{r-1} (W) = \left\{ 0 \right\}.$ 
In both cases \( \int_W j_W^* (\omega \wedge \eta) = \int_W j_W^* \omega \wedge j_W^* \eta = 0\).
\end{proof}
With the help of the previous lemma we get:
\begin{prop}
	Integration induces bilinear forms
	\[ 
			\int : HI_{\pp}^r (M) \times HI_{\qq}^{n-r} (M)  \to \R,
				\bigl ( [\omega], [\eta] \bigr )  \mapsto \int_M \omega \wedge \eta.
\]
\end{prop}
\begin{proof}
Let \( \omega \in \Omega I_{\pp}^r (M) \) and \( \eta \in \Omega I_{\qq}^{n-r} (M) \) be closed forms and \( \omega' \in \Omega I_{\pp}^{r-1} (M) \), \( \eta' \in \Omega I_{\qq}^{n-r-1} (M) \) be any forms. We then apply Lemma \ref{lemma:intexact} to deduce the following.
\begin{align*} 
	\int_M (\omega + d \omega ' ) \wedge \eta & = \int_M \omega \wedge \eta + \int_M \underbrace{d \omega' \wedge \eta}_{= d \left( \omega' \wedge \eta \right)} = \int_M \omega \wedge \eta \quad \text{as well as}\\		\int_M \omega  \wedge ( \eta + d \eta')  & = \int_M \omega \wedge \eta + \int_M  \underbrace{\omega \wedge d \eta'}_{= d\left( \omega \wedge \eta' \right)} = \int_M \omega \wedge \eta.
\end{align*}
\end{proof}

\subsection{Poincar\'e Duality for the Intersection de Rham Complex}
Finally we can state and prove the Poincar\'e duality theorem for \( \OI (M) \):
\begin{thm} (Poincar\'e duality for \( HI \)) \label{thm:PD_OI_isolated} \\
Let $X$ be a compact, oriented Thom-Mather stratified pseudomanifold with filtration $X = X_n \supset X_b \supset \left\{ x_0 \right\}$ and geometrically flat link bundle for the intermediate stratum. Let $M$ be the blowup of $X$ and let $\pp$ and $\qq$ be two complementary perversities. 
 Then, integration induces nondegenerate bilinear forms
\[
		\int : HI_{\pp}^r (M) \times HI_{\qq}^{n-r} (M)  \to \R, \quad
		\bigl ( [ \omega ], [ \eta ] \bigr )  \mapsto \int_M \omega \wedge \eta.
\]
\end{thm}
\begin{proof}
The two short exact sequences
of Lemma \ref{lemma:disttriOI1} induce long exact sequences on cohomology that fit into a diagram of the following form.
\begin{equation} \begin{tikzcd}
		\vdots \ar{d} 							& 		\vdots \ar{d} 		\\
		H^{r-1} \bigl ( \ts \OI (W) \bigr ) \ar{d}{\delta} \ar{r}{\int} & H^{n-r} \bigl ( \tgq \OIq (W) \bigr )^\dagger \ar{d}{(\sigma_W^* \circ c_W^*)^\dagger} \\
		HI^r_{\pp} (M) \ar{d} \ar{r}{\int} 				& 	HI_{\qq}^{n-r} (M)^\dagger \ar{d} \\
		\widetilde{HI}_{\pp}^r (M) \ar{d}{\proj^* \circ j_W^*} \ar{r}{\int} 	& \widetilde{HI}_{\qq}^{n-r} (M,C_W)^\dagger \ar{d}{\Delta^\dagger} \\
		H^r \bigl ( \ts \OI (W) \bigr ) \ar{d} \ar{r}{\int} 		& H^{n-r-1} \bigl ( \tgq \OIq (W) \bigr )^\dagger \ar{d} \\
		\vdots 								& 		\vdots
\end{tikzcd}
\label{eq:commdiagrPDHI}
\end{equation}
We claim that this diagram commutes up to sign.
To show this, we prove step by step that the individual squares in the diagram commute.
We start with the top square (TS) and first describe the connecting homomorphism 
\[ \delta : H^{r-1} \bigl ( \ts \OI (W) \bigr ) \to HI_{\pp}^r (M). \] 
Let \( \omega \in \ts \Omega_{\pp}^{r-1} (W) \) be a closed form and $\proj: \OI (W) = \ts \OI (W) \oplus \tg \OI (W) \to \ts \OI (W)$ the projection to the truncated subcomplex. Then there is a form $\kappa	\in \wOI (M)$ with $\proj (j_W^* \kappa) = \omega.$ We then have, by the standard Zig-zag argument:
\[
	\delta \bigl ( [ \omega ] \bigr ) = [ d \kappa ] \in HI_{\pp}^r (X).
\]
To show that (TS) commutes up to sign, we must show that for \( \omega \in \ts \Omega_{\pp}^{r-1} (W) \) closed and \( \eta \in \Omega I_{\qq}^{n-r} (M) \) closed it holds that
\begin{equation} 
	\int_W \omega \wedge j_W^* (\eta) = \pm \int_M  d \kappa \wedge \eta.
\label{eq:CTS}
\end{equation}
Since \( d \eta = 0 \), \( d \kappa \wedge \eta = d (\kappa \wedge \eta) \) and hence by Stokes' Theorem for manifolds with corners
\[
	\int_M (d \kappa) \wedge \eta = \int_W j_W^* (\kappa \wedge \eta) + \int_E j_E^* (\kappa \wedge \eta).
\]
Since \( \kappa \in \widetilde{\Omega I}_{\pp}^{r-1} (M) \) and \( \eta \in \Omega I_{\qq}^{n-r} (M) \), it holds that
\[
	j_E^* \kappa \in ft_{\geq K} \Omega_{\MS}^r (B) \quad \text{and} \quad j_E^* \eta \in  ft_{\geq K^*} \Omega_{\MS}^{n-r} (B) 
\]
and hence \ref{lemma:tl1} implies that
\[
	\int_E j_E^* (\kappa \wedge \eta) = 0.
\]
What remains is to calculate the integral \( \int_W j_W^* (\kappa \wedge \eta) \):\\
By the definition of $\kappa$, there is a form \( \alpha \in \tg \Omega I_{\pp}^{r-1} (W) \) such that
$ j_W^* (\kappa) = \omega + \alpha$ and hence we can rewrite the integral over $W$ as follows.
\[
	\int_W j_W^* (\kappa \wedge \eta) = \int_W \omega \wedge j_W^* \eta + \int_W \alpha \wedge j_W^* \eta = \int_W \omega \wedge j_W^* \eta
\]
since \( \alpha \in \tg \Omega I_{\pp}^{r-1} (W) \) and \( j_W^* \eta \in \tg \Omega I_{\qq}^{n-r} (W) \) and hence
\[
	\int_W \alpha \wedge j_W^* \eta = 0,
\]
by the same arguments as in the proof of Lemma \ref{lemma:intexact}. Summing up, we have shown that (TS) commutes up to sign.

\vspace{0.2cm}
To prove the commutativity of the bottom square (BS) in (\ref{eq:commdiagrPDHI}), again up to sign, note that
\[
	\Delta: H^{n-r-1} \bigl ( \tgq \OIq (W) \bigr ) \to \widetilde{HI}_{\qq}^{n-r} (M,C_W) 
\] 
is the standard connection homomorphism that belongs to a long exact cohomology sequence induced by a short exact sequence. 
For \( \eta \in \tgq \Omega I_{\qq}^{n-r-1} (W) \) closed, let $\omega \in \widetilde{\Omega I}_{\qq}^{n-r-1} (M)$ be a form with $j_{W}^* \omega = \eta.$ Then, $d \omega \in \widetilde{\Omega I}_{\qq}^{n-r} (M,C_W)$ and $\Delta [\eta] = [d\omega].$
We have to show that for \( \eta \in \tgq \Omega I_{\qq}^{n-r-1} (W) \) and \( \theta \in \widetilde{\Omega I}_{\pp}^r (M) \) closed, with \( \Delta [\eta] = [d \omega] \) as above, it holds that
\[
	\int_M \theta \wedge (d \omega ) = \pm \int_W j_W^* \theta \wedge \eta.
\]
By Stokes' Theorem on manifolds with corners, we get
\[ 
		\pm \int_M \theta \wedge (d \omega) =  \int_M d (\theta \wedge \omega ) = \int_E j_E^* (\theta \wedge \omega) + \int_W j_W^* (\theta \wedge \omega).
	\]
	By definition, \( j_E^* \theta \in ft_{\geq K} \Omega_{\MS}^r (B) \), \( j_E^* \omega \in ft_{\geq K^*} \Omega_{\MS}^{n-r-1} (B). \) Hence, by Lemma \ref{lemma:tl1},
	\[ \int_E j_E^* (\theta \wedge \omega) = \int_E j_E^* \theta \wedge j_E^* \omega = 0. \]
	This implies that
	\[
	\pm \int_M \theta \wedge d \omega = \int_W j_W^* (\theta \wedge \omega) = \int_W j_W^* \theta \wedge \eta.
	\]

	\vspace{0.2cm}
	The middle square in (\ref{eq:commdiagrPDHI}) commutes, since the vertical maps are just inclusions and the horizontal maps both integration of wedge products of two forms.

	\vspace{0.2cm}
	The commutativity of the diagram (\ref{eq:commdiagrPDHI}) (up to sign) together with the fact that the map
	\[
		\int : H^r \bigl ( \ts \OI (W) \bigr ) \to H^{n-r-1} \bigl( \tgq \OIq (W) \bigr)^\dagger 
	\]
	is an isomorphism for all \( r \in \mathbb{Z} \) by Lemma \ref{lemma:PDtrHIW}
	as well as the map
	\[
		\int : \widetilde{HI}_{\pp}^r (M) \to \widetilde{HI}_{\qq}^{n-r} (M,C_W)^\dagger
	\]
	by Proposition \ref{prop:PDonwOI} then enables us to apply the 5-Lemma to conclude the statement of the theorem.
	\end{proof}

\section{Examples} 
We want to give a class of examples of depth two pseudomanifolds, we can apply the intersection space cohomology theory to. We make use of the relation between Thom-Mather stratified pseudomanifolds and compact manifolds with corners and iterated fibration structures, which is described in \cite[Section 2]{SignaturePackageWitt}. In our setting, this is the correspondence between the described depth two pseudomanifolds and $\langle 2 \rangle$-manifolds with one boundary component fibered by a geometrically flat fiber bundle.

We consider flat principal $G$-bundles, with $G$ a compact connected Lie group. 
This example is based on ideas of Laures, see \cite
{Laures}. 

\subsection{Fiberwise Truncation on Principal Bundles}
Before we introduce the examples, we calculate the cohomology of fiberwisely (co)truncated multiplicatively structured forms on flat principal $G$-bundles $p: E \rightarrow B$, with $G$ a compact connected Lie group. 
There are horizontal sections $s: U \rightarrow p^{-1} \left( U \right)$ that induce a trivialization of the bundle with locally constant transition maps by defining local trivializations
\begin{align*}
\phi_U^{-1}: U \times G & \rightarrow p^{-1} (U), \\
(x,g) & \mapsto s(x) \cdot g.
\end{align*}
Byun and Kim prove in \cite[Section 6]{ByunKim} that the Leray-Hirsch Theorem is applicable to flat principal bundles and hence $H^\bullet (E) \cong H^\bullet (B) \otimes H^\bullet (G)$ as algebras. They do so by constructing, for any flat connection $A$ on $E$, an algebra homomorphism $E_A: H_{dR}^\bullet (G) \rightarrow H_{dR}^\bullet (E)$ satisfying $\iota_y^* E_A = \id$ for any inclusion $\iota_y: G \rightarrow p^{-1} \left( p(g) \right), ~ g \mapsto yg.$ The cohomology morphism $E_A$ is induced by a map 
\[
E_A: \mathcal{H}^\bullet (G) \rightarrow \Omega^\bullet (E),
\]
with $\mathcal{H}^\bullet (G) = \left\{ \alpha \in \Omega^\bullet (G): ~ L_g^* \, \alpha = R_g^* \, \alpha = \alpha ~ \forall g \in G \right\}$ the complex of bi-in\-variant forms on $G$, which is isomorphic to $H^\bullet_{dR} (G)$ by \cite[Theorem 12.1]{ChevEilenberg}. If the connection is flat, $E_A (\alpha)$ is a closed form for any $\alpha \in \mathcal{H}^\bullet (G).$
Locally, $E_A(\alpha)$ is given as $E_A(\alpha)|_{p^{-1}(U)} = \kappa_s^* \alpha,$ with $\kappa_s: p^{-1} (U) \rightarrow G$ defined by $y = s\left( p(y) \right) \cdot \kappa_s (y)$ for a horizontal section $s$ on $U$. 
This implies, that $E_A(\alpha)$ is actually a multiplicative form. For let $U$ be any coordinate chart. Then $\kappa_s \circ \phi_U^{-1} (x,g) = \kappa_s \left( s(x) \cdot g \right) = g,$ by the definition of $\kappa_s$. This means that locally $E_A (\alpha)|_{p^{-1} (U)} = \phi_U^* \pi_2^* \alpha$ and hence, $E_A (\alpha)$ is multiplicative.

We want to characterize the cohomology of the complexes of multiplicatively structured truncated respectively cotruncated forms using the complex of bi-invariant forms. To do so, we choose a bi-invariant Riemannian metric $g$ on $G$. That is, a Riemannian metric such that left translations $L_a$ and right translations $R_a$ are isometries for all $a \in G.$ Since $G$ is a compact Lie group, such a bi-invariant Riemannian metric always exists, see \cite[Corollary 1.4]{MilnorCurvature}. 

\begin{prop} \label{prop:TruncCohomPrincBund}
Let $g$ be a bi-invariant Riemannian metric on the compact Lie group $G$ and take the naive truncation and cotruncation of bi-in\-variant forms $\mathcal{H}^r_{<K} (G) = \mathcal{H}^r (G) $ for $r <K$ and zero otherwise and $\mathcal{H}^r_{\geq K} (G) = \mathcal{H}^r (G)$ for $r \geq K$ and zero otherwise.
Then, there are isomorphisms
\[ H^\bullet_{dR} (B) \otimes \mathcal{H}_{<K}^\bullet (G)  \xrightarrow{\cong} H^\bullet \left( \ftsOMS (E)\right) \]
and
\[ H^\bullet_{dR} (B) \otimes \mathcal{H}_{\geq K}^\bullet (G) \xrightarrow{\cong} H^\bullet \left( \ftgOMS (E)\right), \]
defined by 
\[ \left( [\omega], \alpha \right) \mapsto \left[ p^* \omega \wedge E_A^* (\alpha) \right].
\]
\end{prop}
\begin{proof}
The product $p^* \omega \wedge E_A^* (\alpha)$ is always multiplicative since $E_A (\alpha) $ is multiplicative and $p|_U = \pi_1 \circ \phi_U$ for each coordinate chart $U.$ It satisfies the following formula.
\[ \left( p^* \omega \wedge E_A^* (\alpha) \right)|_{p^{-1} (U)} = 
 \phi_U^* \left( \pi_1^* \omega \wedge \pi_2^* \alpha \right). \]

This relation implies that the assignment $\left( [\omega], \alpha \right) \mapsto \left[ p^* \omega \wedge E_A^* (\alpha) \right]$ defines maps as in the statement of the proposition.
This is true because $\mathcal{H}_{<K}^\bullet \subset \tau_{<K} \Omega^\bullet (G)$ and $\mathcal{H}_{\geq K}^\bullet \subset \tau_{\geq K} \Omega^\bullet (G)$. The first inclusion is trivial. 
To see, that $\mathcal{H}_{\geq K}^\bullet $ is a subcomplex of $\tau_{\geq K} \Omega^\bullet (G)$, note, that the left and right translations $L_a, ~ R_a$ are isometries for all $a \in G$, since we work with a bi-invariant metric. Hence, the induced maps on differential forms commute up to sign with the Hodge-star operator, $L_a^* \circ * = \pm * \circ L_a^*,$ $R_a^* \circ * = \pm * \circ R_a^*.$ This implies, that for any $\omega \in \mathcal{H}^\bullet (G)$ it holds that $* \omega \in \mathcal{H}^{\bullet} (G).$ Therefore, since all forms in $\mathcal{H}^\bullet (G)$ are closed, see \cite[(12.3)]{ChevEilenberg}, they are also coclosed and hence harmonic. That implies that $\mathcal{H}^r (G) \subset \text{Harm}^r (G) \subset \ker d^*$ and therefore, it holds that $\mathcal{H}_{\geq K}^\bullet (G) \subset \tau_{\geq K} \Omega^\bullet (G).$

The same Mayer-Vietoris type argument which is used to prove the Leray-Hirsch Theorem (together with the Poincar\'e Lemmas for fiberwise (co)\-truncated multiplicatively structured forms of \cite{BandR}) then implies the statement of the proposition.
\end{proof}

\subsection{Principal bundles over \texorpdfstring{$\langle 2 \rangle$}{two}-manifolds}
Let $G$ be a compact Lie group and $Q$ a $\langle 2 \rangle$-manifold with boundary $\dQ = \dB \cup_{\partial (\dB)} \dB'$ such that there is a group represantation $\rho: \pi_1 (Q) \to G$ of the fundamental group of $Q$. Let $\pi: \widetilde{Q} \rightarrow Q$ denote the universal covering of $Q$. Then, the map $p: M_{\rho} := \widetilde{Q} \times_\rho G \rightarrow Q, ~ \left( x, g \right) \mapsto \pi \left( x \right),$ defines a flat principal $G$-bundle.
$M_\rho$ is a $\langle 2 \rangle$-manifold as well and can be interpreted as the blowup of a pseudmanifold with three strata. This pseudmanifold is constructed by first blowing down $p^{-1} \left( \dB \right)$ fiberwisely and then blowing down the boundary of the resulting pseudmanifold with boundary in the sense of \cite[Section 2]{SignaturePackageWitt}.
\begin{ex}
We first construct a (high dimensional) orientable closed manifold $P$ with fundamental group $\pi_1 (P) = \Z/2\Z$. Let $n\geq 4, ~ k \geq 2 $ and set $N := n+k.$
Starting with the manifold $S^{1} \times S^{n-1},$ let $c$ be a closed embedded curve such that its homotopy class represents two times the generator of the fundamental group $\pi_1 \left( S^{1} \times S^{n-1} \right).$ If we cut out a tubular neighbourhood $c \times D^{n-1}$ of $c$ and glue a disk to $c$, we get the space
\[
P = \left( (S^1 \times S^{n-1} ) \setminus (c \times D^{n-1}) \right) \cup_{\partial D = c} (D^2 \times S^{n-2}),
\]
which is a closed oriented $n$-dimensional manifold. By the Seifert-van Kampen theorem, the fundamental group of $P$ is $\pi_1 (P) = \Z/2\Z$ and, by a Mayer-Vietoris computation, $H_* (P; \R) = H_* (S^n; \R).$

The manifold $Q$ is then obtained by cutting an Euclidean ball $\mathcal{B}$ out of $P \times S^k, ~ Q := P \times S^k \setminus \mathcal{B}.$ Depending on the embedding of $\mathcal{B}$ in $P \times S^k$ and the associated collars, we can view $Q$ either as a manifold with boundary $\partial \mathcal{B} = S^{N-1}$ or as a $\langle 2 \rangle$-manifold with decomposed boundary $\partial \mathcal{B} = S^{N-1} = D_1^{N-1} \cup_{S^{N-2}} D_2^{N-1}$. The Seifert-van Kampen theorem implies, that the fundamental group of $Q$ still is $\pi_1 (Q) = \Z/2\Z,$ since $ \pi_1 (Q) \cong \pi_1 ( P \times S^k) \cong \pi_1 (P) \times \pi_1 (S^k) = \pi_1 (P) = \Z/2\Z.$ Hence, there is a group homomorphism $\rho: ~\pi_1 (Q) \cong \Z/2\Z \hookrightarrow S^1 \hookrightarrow S^1 \times S^1 = T^2,$ with $S^1  \hookrightarrow S^1 \times S^1$ any embedding, for example inclusion as first factor at some point in the second factor. Let $\pi: \widetilde{Q} \rightarrow Q$ be the universal covering of $Q$. The homomorphism $\rho$ defines a flat principal $T^2$-bundle
\[
p: M_\rho := \widetilde{Q} \times_{\rho} T^2 \rightarrow Q, ~ [\widetilde{x},t] \mapsto \pi \left( \widetilde{x} \right).
\]
The total space of this bundle can be seen as the blowup of a $2$-strata pseudomanifold with singular set $\dQ$ and link bundle $p|: p^{-1} (\dQ) \rightarrow \dQ.$ Alternatively, one can interpret $Q$ as the blowup of a $3$-strata pseudomanifold with filtration by singular sets $X \supset \dQ \cong S^{N-1} \supset \{ x_{0} \},$ with $x_0 \in S^{N-1}$ some point. 

We first calculate the intersection space cohomology of the associated $2-$strata pseudomanifold $Y$. We use the long exact cohomology sequence induced by Banagl's short exact sequence
\begin{equation}
\begin{tikzcd}[column sep=small]
0 \ar{r} & \Omega^\bullet \left(M_\rho, C_{\dM_\rho} \right) \ar{r}  & \OI (M_\rho) \ar{r}{j^*} & \ftgOMS \left( \dM_\rho \right) \ar{r} & 0 
\end{tikzcd}
\label{eq:ses_HI_2strata}
\end{equation}
where $\dM_\rho = p^{-1} \left( \dQ \right)$ and $C_{\dM_\rho}$ is a collar neighbourhood of the boundary (see \cite[(13)]{BandR}, where the sequence is stated as a distinguished triangle). Proposition \ref{prop:TruncCohomPrincBund} implies, that 
\begin{align*}
H^\bullet \left( \ftgOMS \left( \dM_\rho \right) \right) & \cong H^{\bullet} (\dQ) \otimes \mathcal{H}_{\geq K}^\bullet (T^2) \\
& \cong \mathcal{H}_{\geq K}^\bullet (T^2) \oplus \mathcal{H}_{\geq K}^\bullet (T^2) \, [1-N].
\end{align*}
By Poincar\'e-Lefschetz duality, we have $H^r (Q) \cong H_{N-r} \left( P \times S^k, \mathcal{B} \right).$ Since $\mathcal{B}$ is contractible, these are the reduced homology groups of $P \times S^k.$ By the K\"unneth Theorem, $\widetilde{H}_* (P \times S^k)$ is zero in all degrees but $n,k,N$ and isomorphic to $\R$ for $r =  k, n, N$ for $n \neq k$ while $H_n (P \times S^k) \cong \R \otimes \R$ if $n=k$.  Hence, the cohomology of the principal flat $T^2-$bundle $M_\rho$ satisfies the following relation.
\[
H^\bullet (M_\rho) \cong H^\bullet (T^2) \oplus H^\bullet (T^2) \, [-k] \oplus H^\bullet (T^2) \, [-n].
\]
The cohomology of the boundary satisfies the similar relation $H^\bullet \left( \dM_\rho \right) \cong H^\bullet (S^N) \otimes H^\bullet (T^2) \cong H^\bullet (T^2) \oplus H^\bullet (T^2) \, [1-N].$
The pullback to the boundary $j^*: \Omega^\bullet (M_\rho) \to \Omega^\bullet (\dM_\rho)$ induces a map on cohomology which maps the non-shifted comohomology $H^\bullet (T^2)$ to $H^\bullet (T^2)$ and is trivial in all other degrees, by dimensional reasons. Together with the long exact sequence of the pair $\left( M_\rho, \dM_\rho \right),$ this gives rise to the following formula.
\begin{align*}
H^\bullet \left( M_\rho, \dM_\rho \right) 
& \cong H^\bullet (T^2) \, [-k] \oplus H^\bullet (T^2) \, [-n] \oplus H^\bullet (T^2) \, [-N].
\end{align*}
Since there is an injection $H^\bullet \left( \ftgOMS ( \dM_\rho ) \right) \hookrightarrow H^\bullet ( \dM_\rho ),$ also by Proposition \ref{prop:TruncCohomPrincBund}, using the long exact sequence induced by (\ref{eq:ses_HI_2strata}) from above, the intersection space cohomology groups of $Y$ are given as
\begin{align*}
HI_{\pp}^r (M_\rho) & \cong H^r (M_\rho, \dM_\rho) \, / \, \delta \left(H^{r-1} (\ftgOMS (\dM_\rho)) \right) \\
& \quad \oplus \ker \left( \delta: H^{r} (\ftgOMS (\dM_\rho)) \to H^{r+1} (M_\rho, \dM_\rho) \right) \\
\Rightarrow ~ HI_{\pp}^\bullet (M_\rho) &\cong H^\bullet (T^2) \, [-k] \oplus H^\bullet (T^2) \, [-n] \oplus \mathcal{H}_{<K}^\bullet (T^2) \, [-N] \oplus \mathcal{H}_{\geq K}^\bullet (T^2),
\end{align*}
where $\delta: H^{\bullet} (\dM_\rho) \to H^{\bullet +1} (M_\rho, \dM_\rho)$ denotes the connecting homomorphism of the long exact cohomology sequence of the pair $(M_\rho, \dM_\rho)$.
Note, that the Poincar\'e duality isomorphism between complementary perversities interchanges the first two factors and the latter two, respectively.

\vspace{2ex}
Let us also calculate the intersection space cohomology groups of the associated $3-$strata pseudomanifold $X.$ The approach to calculate these is geared to our proof technique for the Poincar\'e duality theorem. In the notation from the previous sections, the manifold with boundary is $M = M_\rho,$ with boundary parts $E= p^{-1} (D_1)$ and $W = p^{-1} (D_2)$ and corner $\dE = \dW = p^{-1} (\partial D_1) = p^{-1} (S^{N-2})$. The base of the link bundle $p: E \to B$ is $B = D_1$ with boundary $\dB = S^{N-2}.$
The two cotruncation values that are needed on the respective boundary parts are $K := 2- \pp(3)$ and $L := N-1 - \pp (N)$.
We calculate the cohomology groups $\widetilde{HI}_{\pp} (M_\rho)$ first, using the following distinguished triangle. 
\[
\begin{tikzcd}
\Omega^\bullet (M, C_E) \ar{r} & \wOI (M) \ar{r} & \ftgOMS (B) \ar{r} & \Omega^\bullet (M, C_E)[+1]
\end{tikzcd}
\]
This distinguished triangle exists, since the kernel of the surjective pullback map $\wOI \to P_{\geq K}^\bullet (B)$, where the latter complex is defined in Definition \ref{def:Pcomplexes}, is the complex
\[
\left\{ \omega \in \Omega^\bullet (M, C_E) \; | \; c_{W}^* \omega = \pi_W^* \eta ~ \text{for some}~ \eta \in \Omega^\bullet (W) \right\}.
\]
This complex is quasi-isomorphic to $\Omega^\bullet (M, C_E)$ by the same arguments as in Lemma \ref{lemma:MSdc}. Since $P^\bullet_{\geq K} (B)$ is quasi-isomorphic to $\ftgOMS (B)$ by the referenced lemma, we have the above distinguished triangle.

The long exact sequence of the pair $(M,E) = \left( M_\rho, p^{-1} (D_1) \right)$, together with the previous calculations, implies that $ H^\bullet \left( M_\rho, p^{-1} (D_1) \right) \cong H^\bullet (T^2) \, [-k] \otimes H^\bullet (T^2) \, [-n].$ Since the $(n-1)$-dimensional disc $D_1$ is contractible, Proposition \ref{prop:TruncCohomPrincBund} gives $H^\bullet \left( \ftgOMS (p^{-1} (D_1) ) \right) \cong \mathcal{H}_{\geq K}^\bullet (T^2).$ 
Therefore, by dimensional reasons, the connecting homomorphism in the long exact cohomology sequence induced by the above triangle is trivial and hence the cohomology groups $\widetilde{HI}^\bullet_{\pp} (M_\rho)$ are a direct product,
\begin{align*}
\widetilde{HI}_{\pp}^\bullet (M_\rho) &\cong H^\bullet \left( M_\rho, p^{-1} (D_1) \right) \oplus H^\bullet \left( \ftgOMS \left( p^{-1} (D_1) \right) \right) \\
& \cong H^\bullet (T^2) \,[-k] \oplus H^\bullet (T^2) \,[-n] \oplus \mathcal{H}_{\geq K} (T^2).
\end{align*}
To calculate the intersection space cohomology groups $HI_{\pp}^\bullet (X),$ we use the second short exact sequence of Lemma \ref{lemma:disttriOI1}, which we recall below adapted to this setting.
\[
0 \rightarrow \OI (M_\rho) \xrightarrow{\iota} 	 \wOI (M_\rho) \xrightarrow{\proj \circ j^*} \ts \OI (p^{-1} (D_2)) \rightarrow 0
\]
 To calculate the groups $HI_{\pp}^\bullet \left( p^{-1} (D_2) \right),$ we use the same argumentation as in the first part of this example.  
\begin{align*} 
H^\bullet \left( p^{-1} (D_2), p^{-1} (S^{N-2}) \right) &\cong H^\bullet (T^2) \,[1-N] \quad \text{and} \\
H^\bullet \left( \ftgOMS \left( p^{-1} (S^{N-2} ) \right) \right) & \cong H^{\bullet} (S^{N-2}) \otimes \mathcal{H}_{\geq K}^\bullet (T^2).
\end{align*}
Together with the long exact cohomology sequence induced by (\ref{eq:ses_HI_2strata}), adapted to this setting, we get
\[
HI_{\pp}^\bullet \left( p^{-1} (D_2)\right) \cong \mathcal{H}_{\geq K}^\bullet (T^2) \oplus \mathcal{H}_{<K}^\bullet (T^2) \,[1-N].
\]
Since $L = N-1 - \pp(N) \leq N-1,$ cotruncating below that degree gives
\[
H^\bullet \left (\tau_{<L} \OI \left( p^{-1} (D_2) \right) \right)
\cong \mathcal{H}_{\geq K}^\bullet (T^2).
\]
Note, that this is the image of the factor $\mathcal{H}_{\geq K}^\bullet (T^2)$ in our depiction of $\widetilde{HI}_{\pp}^\bullet (M_\rho)$ under the cohomology map induced by the pullback $j^*: \wOI (M_\rho) \to \OI (p^{-1} (D_2))$ followed by the projection to the truncated complex. Hence, the map
\[ \proj^* \circ j^*: \widetilde{HI}_{\pp}^\bullet (M_\rho) \to H^\bullet \left( \tau_{<L} \OI (p^{-1} (D_2)) \right) \]
is surjective.
Therefore, the connecting homomorphism in the long exact cohomology sequence induced by the second short exact sequence of Lemma \ref{lemma:disttriOI1}, which we recalled above, is trivial and we get the following result.
\[
HI_{\pp}^\bullet (M_\rho) \cong \ker (\proj^* \circ j^*)
\cong  H^\bullet (T^2) \,[-k] \oplus H^\bullet (T^2) \,[-n],
\]
which does not depend on the chosen perversity. 
In particular, the groups are different from the ones calculated in the first part of the example.
The Poincar\'e duality isomorphism interchanges the two factors in the direct sum. 
\end{ex}

\bibliographystyle{alpha}
\bibliography{bibtex}
\end{document}